\documentclass[12pt]{article}

\usepackage{amsmath,amssymb,amsthm}
\usepackage[usenames,dvipsnames,table]{xcolor}
\usepackage{tikz,bbm,diagbox}
\usepackage[colorlinks=true,citecolor=black,linkcolor=black,urlcolor=blue]{hyperref}

\usetikzlibrary{decorations.markings}
\usetikzlibrary{calc}
\usepackage[inline]{enumitem}
\usepackage{float}

\allowdisplaybreaks

\renewcommand*{\Pr}{\mathop{\mathrm{Pr}}}
\newcommand*{\E}{\mathrm{\mathbf{E}}}

\newtheorem{theorem}{Theorem}[section]
\newtheorem{lemma}[theorem]{Lemma}

\newtheorem{conjecture}[theorem]{Conjecture}

\def\nfrac#1#2{{\textstyle\frac{#1}{#2}}}
\def\dfrac#1#2{\lower0.15ex\hbox{\large$\textstyle\frac{#1}{#2}$}}

\newcommand{\xvec}{\boldsymbol{x}}
\newcommand{\bvec}{\boldsymbol{b}}
\newcommand{\yvec}{\boldsymbol{y}}
\newcommand{\avec}{\boldsymbol{a}}
\newcommand{\Bvec}{\boldsymbol{B}}
\newcommand{\evec}{\boldsymbol{e}}
\newcommand{\uvec}{\boldsymbol{u}}
\newcommand{\vvec}{\boldsymbol{v}}
\newcommand{\wvec}{\boldsymbol{w}}
\newcommand{\x}{\tau}   

\def\kind#1{{k^+(#1)}}  
\def\kSSCM#1{{k^-(#1)}} 
\def\kactual#1{{k^*(#1)}}  
\def\kSSCM#1{{k^-(#1)}} 

\title{Decomposing random regular graphs into stars}

\author{
Michelle Delcourt 
\thanks{Department of Mathematics, Toronto Metropolitan University, Toronto, Ontario M5B 2K3, Canada. {\tt mdelcourt@torontomu.ca}\ . 
}
\and 
Catherine Greenhill\thanks{School of Mathematics and Statistics, UNSW Sydney, NSW 2052, Australia. \texttt{c.greenhill@unsw.edu.au}\ .}
\and 
Mikhail Isaev\thanks{School of Mathematics and Statistics, UNSW Sydney, NSW 2052, Australia. \texttt{m.isaev@unsw.edu.au}\ . 
}
\and 
Bernard Lidick{\' y}\thanks{Department of Mathematics, Iowa State University, Ames, IA, USA. \texttt{lidicky@iastate.edu}\ .  
}
\and 
Luke Postle
\thanks{Combinatorics and Optimization Department, University of Waterloo, Waterloo, 
Ontario N2L 3G1, Canada. \texttt{lpostle@uwaterloo.ca}\ . 
}
}  

\date{29 June 2025}

\begin{document}

\maketitle

\begin{abstract}
We study $k$-star decompositions, that is,
partitions of the edge set into disjoint stars with $k$ edges, 
in the uniformly random $d$-regular graph model $\mathcal{G}_{n,d}$.
Using the small subgraph conditioning method, we prove an existence result for such decompositions for all $d,k$ such that 
$d/2 < k \leq d/2 + \max\{1,\nfrac{1}{6}\log d\}$.
More generally, we give a sufficient existence condition 
that can be checked
numerically for any given values of $d$ and $k$. Complementary negative
results are obtained using the independence ratio of random regular graphs.
Our results establish an existence threshold for $k$-star decompositions in $\mathcal{G}_{n,d}$ for all  $d\leq 100$ and $k > d/2$. 

For smaller values of $k$, the connection between $k$-star decompositions
and $\beta$-orientations allows us to apply results of 
Thomassen (2012) and Lov{\' a}sz, Thomassen, Wu and Zhang (2013). 
We prove that random $d$-regular graphs satisfy their assumptions
with high probability, thus establishing a.a.s.\ existence of $k$-star decompositions
(i) when $2k^2+k\leq d$, and (ii) when $k$ is odd and $k < d/2$.
\end{abstract}

\section{Introduction}\label{s:introduction}

A $k$-\emph{star} is a complete bipartite graph $K_{1,k}$, and a $k$-{star
decomposition} of a given graph $G$ is a partition of the edge set of $G$
into disjoint $k$-stars.
The problem of decomposing graphs into disjoint stars has been
well-studied in the design community: see, for 
example,~\cite{CH,hoffman,tarsi,YISUH} and references therein. 

We consider the problem of existence of $k$-star decompositions in
the uniformly random $d$-regular graph $\mathcal{G}_{n,d}$, where $k$, $d$
are constants and $n\to\infty$. 
The case $(d,k)=(4,3)$ was studied by the first and last author in~\cite{DP}.
A necessary condition for existence is that $k$ divides $dn/2$,  
since there are exactly $dn/(2k)$ disjoint $k$-stars in any such decomposition.
Let $\mathcal{N}_{d,k}$ be defined by
\[ \mathcal{N}_{d,k}:= \{ n\in\mathbb{Z}^+ \mid 2k \text{ divides } dn\}.\]
All our asymptotic statements are with respect to $n\to \infty$
along $\mathcal{N}_{d,k}$.
If an event holds with probability which tends
to~1, then we say that the event holds \emph{asymptotically almost
surely}  abbreviated as a.a.s. 

\subsection{Connection to the independence ratio}\label{s:independence-ratio}

We find the case $k>d/2$ the most interesting due to its connection with
the independence ratio, which is an important quantity in
statistical physics and combinatorial optimisation.
Recall that the \emph{independence ratio} of a graph is the size of its largest independent set divided by the number of vertices. 
To describe the connection to our problem, observe that every vertex is the centre of at most one star when $k>d/2$.
Given a $k$-star decomposition, we say that a given vertex is a \emph{centre} if it is the centre of a star, and otherwise it is called a \emph{leaf}.
There are $\dfrac{dn}{2k}$ centres and $\dfrac{(2k-d)n}{2k}$ leaves in any decomposition.
Observing that the leaves form an independent set, we establish another necessary condition:
\begin{equation}\label{connection-leaves}
	\begin{aligned}
	&\text{if a $d$-regular graph $G$ admits a $k$-star decomposition  
 }  \\[-0.4ex] &\text{then $G$ contains an independent set of size $\dfrac{(2k-d)n}{2k}$.}
	\end{aligned}
\end{equation}

Bayati, Gamarnik and Tetali~\cite{BGT} proved that the
independence ratio of $\mathcal{G}_{n,d}$ converges to a constant $\alpha^*(d)$ when $d\geq 3$.  
For a history of work on the independence ratio of $\mathcal{G}_{n,d}$, see for example~\cite{DSS,Wormald-survey}. 
It follows from (\ref{connection-leaves}) that
\begin{equation}
\label{eq:indset-nonexistence}
\text{ if $k > \frac{d}{2(1-\alpha^*(d))}$ then $\mathcal{G}_{n,d}$ has no $k$-star 
               decomposition a.a.s.}
\end{equation}
Furthermore, replacing $\alpha^*(d)$ by any upper bound in (\ref{eq:indset-nonexistence})
gives a non-existence result.
Frieze and {\L}uczak~\cite{FL} proved that $\alpha^*(d)$ is 
bounded above by $\dfrac{2\log d}{d}$.
When $d\geq 20$ it is believed that 
$\alpha^*(d)$ is precisely the value $\alpha^{\text{1-RSB}}(d)$ predicted by the one-step replica
symmetry breaking (1-RSB) heuristic of statistical physics (see~\cite{BKZZ, harangi}).
Ding, Sly and Sun~\cite{DSS} gave a rigorous proof of this
when $d\geq d_0$ is sufficiently large.
Lelarge and Oulamara~\cite{LO} used the interpolation method to prove that 
the 1-RSB prediction always gives an upper bound on $\alpha^*(d)$, and
provided explicit values of $\alpha^{\text{1-RSB}}(d)$ for $d=3,\ldots, 10$.
Harangi~\cite{harangi} used $r$-step replica symmetry breaking formulas 
to obtain
improved upper bounds for the independence ratio for $d=3,\ldots, 19$.

For small values of $d$, there is a constant gap between $\alpha^{\text{1-RSB}}(d)$ and the 
best-known lower bounds on $\alpha^*(d)$.
Furthermore, as remarked in~\cite{DSS}, the 1-RSB prediction is believed
to fail on low-degree graphs. 
Using the differential equations method,
lower bounds on $\alpha^*(d)$ have been obtained by Wormald~\cite{wormald95},
with some improvements by Duckworth and Zito~\cite{DZ}. 
Experimental (non-rigorous) lower bounds can also be found in~\cite{MK}.
Taking the contrapositive of (\ref{connection-leaves}), we observe that
\begin{equation}
\label{eq:indset-existence}
\text{if \, $\mathcal{G}_{n,d}$\, a.a.s.\ has a $k$-star decomposition then}\,\, \alpha^*(d)\geq 1 - \frac{d}{2k}.
\end{equation}
For example, in Section~\ref{ss:small-values} we show that $\mathcal{G}_{n,16}$ a.a.s.\ has a 10-star-decomposition. Then using (\ref{eq:indset-existence}), 
we obtain a new lower bound of $1/5$ on the independence
ratio of $\mathcal{G}_{n,16}$. 
This improves on the best-known explicit lower bound of 0.1985 
given by Wormald~\cite[Table~1]{wormald95}, obtained using the
differential equations method. 

\subsection{Threshold conjecture and supporting results}\label{s:main-results}

For $d\geq 3$, let $\kind{d}$ be the largest positive integer 
$d/2 < k < d$ such that
\begin{equation}\label{eq:k+}
 d^{d-1}
	 > (2k-d)^{(2k-d)/d}\,
	 2^{(d^2-2k)/d}\, k^{k(d-2)/d}\, (d-k)^{d-k}.
 \end{equation}
Note that $\kind{d}$ is well defined as (\ref{eq:k+}) holds when $k=\lceil d/2\rceil$. 
Furthermore, substituting $k=d/2 + c\log d$ into (\ref{eq:k+}) and solving
for $c$, we find that $c=1 + O\left(\frac{\log\log d}{d}\right)$. From this
	we can conclude that 
\begin{equation}
\label{eq:k+asymptotics}
\kind{d} = d/2 + \log d + O\left(\log\log d\right),
\end{equation}
using a concavity argument to exclude larger values of $k$; see (\ref{eq:monotonicity}) for more details.
The implicit constant in the above $O(\cdot)$ term is independent of $d$.
As we will see in Section~\ref{sec:ind-negative}, the definition of $\kind{d}$ corresponds to the
expectation threshold for the a.a.s.\ existence of independent sets of size $(2k-d)n/(2k)$ in
$\mathcal{G}_{n,d}$.

\medskip

If $\mathcal{G}_{n,d}$ a.a.s.\ has a $(k+1)$-star decomposition, it does not
obviously follow that $\mathcal{G}_{n,d}$ a.a.s.\ has a $k$-star decomposition,
even if we assume all necessary divisibility conditions.
Despite this,
we believe that there is a threshold value $\kactual{d}$ which is close to
$\kind{d}$.

\begin{conjecture}
\label{our-conjecture}
Let $d>k\geq 2$ be fixed positive integers. There exists $\kactual{d}$ 
such that
\[ \Pr\big(\mathcal{G}_{n,d}  \text{ has a $k$-star decomposition}\big)\rightarrow \begin{cases} 1 & \text{ if $k \leq \kactual{d}$,}\\ 
0 & \text{ if $k > \kactual{d}$.} \end{cases}
\]
Furthermore, the threshold value satisfies $\kactual{d}\in\{ \kind{d}-1,\, \kind{d}\}$.
\end{conjecture}

All our results support this conjecture.

\medskip

First we consider small values of $k$, that is, $2\leq k\leq d/2$.
It was asserted in~\cite{DP} that the work of Lov{\' a}sz, Thomassen, Wu and Zhang~\cite[Theorem~3.1]{LTWZ} implies the following: if $k \leq \lceil d/2\rceil$, then a.a.s. $\mathcal{G}_{n,d}$ admits a $k$-star decomposition.  In this paper, we would like to clarify that the situation is not as straightforward, since~\cite[Theorem~3.1]{LTWZ} assumes that $k$ is odd and we require a stronger assumption than $d$-edge-connectivity.  Thus we must perform a more careful application of~\cite[Theorem~3.1]{LTWZ}, together with some known properties of random regular graphs, to prove the following in Section~\ref{s:small-k}.

\begin{theorem}
\label{thm:small-k}
Let $k,d$ be positive integers with $k < d$. 
Then a.a.s.\ $\mathcal{G}_{n,d}$ has a $k$-star decomposition in the following
cases:
\begin{itemize}
\item[\emph{(i)}] $k=2$ or $d$ is a multiple of $2k$. 
\item[\emph{(ii)}] $d\geq 2k^2+k$.
\item[\emph{(iii)}] $k$ is odd and $3\leq k < d/2$.
\end{itemize}
\end{theorem}

We believe that the a.a.s.\ existence for remaining unknown values with $k\leq d/2$
can be resolved by extending the result of~\cite{LTWZ} to even~$k$.
For random regular graphs, it may also be possible to use the Expander Mixing Lemma 
to achieve this (when the degree is sufficiently large), as was done by Alon and Pra{\l}at for the Jaeger conjecture in~\cite{AP}.

\medskip

Now we turn to the case $k>d/2$.
The following non-existence result is established using the first moment approach
in Section~\ref{sec:ind-negative}.

\begin{theorem}
\label{thm:independence}
Suppose that $d\geq 3$. 
If $(d,k)=(5,4)$ or $k > \kind{d}$, then 
\[ \Pr\big(\mathcal{G}_{n,d}  \text{ has a $k$-star decomposition}\big)
\rightarrow 0.\]
\end{theorem}

\medskip

\noindent 
The next result is an extension of the work~\cite{DP} for the case $(d,k)=(4,3)$.
We prove it using the small subgraph 
conditioning method~\cite{janson,RW92}, see Section~\ref{s:SSCM} and Section~\ref{s:applying}.
\begin{theorem}
Let $d\geq 4$ and $d/2 < k \leq d/2 + \max\left\{1,\, \dfrac{1}{6}\log d\right\}$.
Then
\[ \Pr\big(\mathcal{G}_{n,d}  \text{ has a $k$-star \emph{decomposition}}\big)\rightarrow 1.\]
\label{thm:main2}
\end{theorem}

\vspace*{-\baselineskip}

The constant $\frac{1}{6}$ is an artifact of our proof. We believe that the result should be true with $\frac{1}{6}$ replaced by something close to 1, 
which would match (\ref{eq:k+asymptotics}).
Finally, in Section~\ref{s:numerical} we provide an implicit sufficient condition for the a.a.s.\ existence
of $k$-star decompositions. 

\subsection{Numerically-verifiable sufficient condition for existence }\label{s:numerical}

Define the polynomial $f$ and function $\eta$ on $[0,\infty$) by
   \begin{align}
   	f(x) &= \frac{1}{\binom{d}{k}}\, \sum_{i=0}^{d-k} \binom{k}{i} \binom{d-k}{i} x^{k-i}, \label{expan0}\\
 \eta(x) &=   \frac{2k-d}{d-k + \sqrt{(d-k)^2 + d(2k-d)\, f(x)}}.
   \label{eq:eta-def} 
   \end{align}
Note that $f(x)$ is  the PGF of the hypergeometric distribution with parameters $(d,k,k)$.  
For $d\geq 3$ we define $\kSSCM{d}$ to be the largest 
 integer such that for $k$ satisfying
 $d/2 < k < \kSSCM{d}$ we have that
\begin{equation}
\label{eq:P2}
   (2k-d)^2 <4k-d-2
   \end{equation}
   and
\begin{equation}\label{eq:P1}
	 \begin{aligned}
		 (1+\eta(x)) f(x)  = \frac{(x+1)\, f'(x)}{k} \quad \text{has unique solution $x=1$} 
		 \\
		 \text{on} \quad
		 \left(\left(1+\dfrac{(2k-d)^2 d}{k (d-k) (4k-d-2-(2k-d)^2)}\right)^{-1}, \dfrac{5k-2d}{d-k}\right).
	\end{aligned}
\end{equation}
The condition (\ref{eq:P1}) can be checked numerically for
any specific pair $(d,k)$;  see Subsection~\ref{ss:small-values} for an example.
Direct computations show that $k = \lceil  \frac{d+1}{2}\rceil$ 
satisfies~\eqref{eq:P2}, and the estimates of Section~\ref{ss:proof-B-k0} 
show that~\eqref{eq:P1} also holds for this value of $k$.
On the other hand, (\ref{eq:P2}) fails for all $k\geq d$ which shows that 
$\kSSCM{d}$ is well-defined.

\bigskip

\noindent Our final main result is the following theorem, which will be proved
together with Theorem~\ref{thm:main2} using 
the small subgraph conditioning method; 
see Section~\ref{s:SSCM} and Section~\ref{s:applying}.

\begin{theorem}
\label{thm:main}
If $d\geq 3$ and $d/2 < k \leq \kSSCM{d}$, then
\[ \Pr\big(\mathcal{G}_{n,d}  \text{ has a $k$-star decomposition}\big)\to 1.\]
\end{theorem}

\medskip
Using this condition we have numerically verified Conjecture~\ref{our-conjecture} for $3\leq d\leq 200$ and $k\geq d/2$. In particular, we discovered that $\kSSCM{d}=\kind{d}$ for 
more than 60\% of the values of $d$; in these cases there is no gap between our existence
and non-existence results, confirming the existence of a threshold $k^*(d)$.
Furthermore, for $d\in \{10,21,42,44,83,85,158,160,162\}$, we can prove that
$k^*(d)=\kSSCM{d}=\kind{d}-1$ using the value of $\alpha^{\text{1-RSB}}(d)$ and
(\ref{eq:indset-nonexistence}).

\begin{figure}[ht!]
\begin{center}
\begin{tikzpicture}
\node at (0,0) {\includegraphics[scale=0.28]{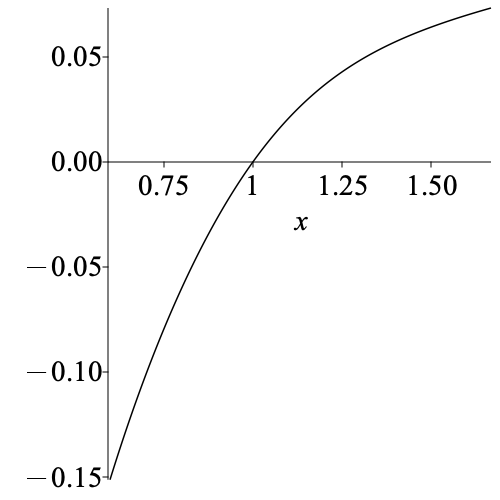}};  
\node at (7,0) {\includegraphics[scale=0.28]{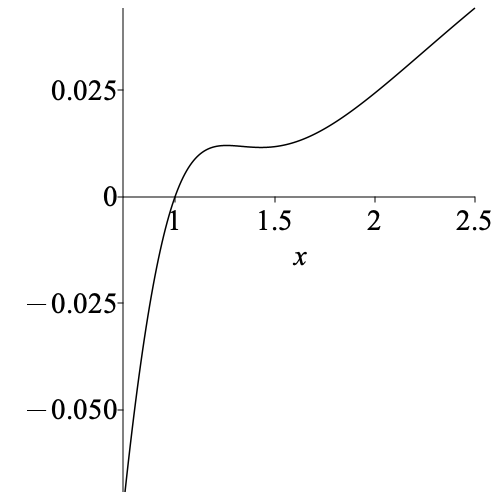}};  
\node at (0.5,-3.0) {$d=20$,\, $k=11$};
\node at (7.5,-3.0) {$d=20$,\, $k=12$};
\end{tikzpicture}
\caption{Plots of $\widehat{f}(x)$, defined in (\ref{eq:widehatf}), for $d=20$ and $k=11,12$.} 
\label{f:plots}
\end{center}
\end{figure}

As an illustration, we show how to compute $\kSSCM{d}$ when \mbox{$d=20$}.
Figure~\ref{f:plots} contains plots of the function 
\begin{equation}
\label{eq:widehatf}
\widehat{f}(x):= \frac{k(1+\eta(x))\, f(x)}{(x+1)\, f'(x)} - 1
\end{equation}
on the range specified in (\ref{eq:P1}) for $d=20$ and $k=11,12$. 
Since $\widehat{f}(x)=0$
if and only if (\ref{eq:P1}) holds, and since (\ref{eq:P2})
holds for all these values of $(d,k)$, these plots prove that $\kSSCM{20}=12$. 

\medskip

Direct computations show that $\kind{20}=12$. 
Since $\kSSCM{20}=\kind{20}$, it follows that $k^*(20)=12$. Combining 
Theorem~\ref{thm:small-k}, Theorem~\ref{thm:independence} and Theorem~\ref{thm:main} we conclude that
Conjecture~\ref{our-conjecture} holds when $d=20$ for all $k=2,\ldots, 19$ except possibly
for $k=8$.

\subsection{Small values of $d$ and $k$}\label{ss:small-values}

In Table~\ref{table:values} 
and Table~\ref{table:big} we summarise what is known about the a.a.s.\ existence or non-existence of $k$-star decompositions of $\mathcal{G}_{n,d}$. 

\renewcommand{\arraystretch}{1.2}

\begin{table}[ht!]
\begin{center}
\begin{tabular}{|c|c|c|c|c|c|c|c|c|c|c|c|c|c|c|c|c|c|c|}
\hline
$d$ & 3 & 4 & \cellcolor{gray!15}5 & 6 & 7 & 8 & 9 & \cellcolor{gray!15}10 & 11 & \cellcolor{gray!15}12 & 13 & \cellcolor{gray!15}14 & 15 & 16 & 17 & 18 & 19 & 20\\
\hline
$\kSSCM{d}$ & 2 & 3 & \cellcolor{gray!15}3 & 4 & 5 & 5 & 6 & \cellcolor{gray!15}6 & 7 & \cellcolor{gray!15}7 & 8 & \cellcolor{gray!15}8 & 9 & 10 & 10 & 11 & 11 & 12\\
\hline
$\kind{d}$ & 2 & 3 & \cellcolor{gray!15}4 & 4 & 5 & 5 & 6 & \cellcolor{gray!15}7 & 7 & \cellcolor{gray!15}8 & 8 & \cellcolor{gray!15}9 & 9 & 10 & 10 & 11 & 11 & 12\\
\hline
\end{tabular}
\caption{The values of $\kSSCM{d}$ and $\kind{d}$ for $d=4,5, \ldots, 20$.}
\label{table:values}
\end{center}
\end{table}

Table~\ref{table:values} shows that for $d=3,\ldots, 20$ we have $\kSSCM{d}=\kind{d}$ unless $d\in\{5,10,12,14\}$.
Using these values and our results, we can fill in most entries of Table~\ref{table:big}
for $3\leq d\leq 20$ and $2\leq k < d$. 
The highlighted cell in each row is $(d,\kSSCM{d})$.  Every 
cell strictly to the right of $(d,\kind{d})$ is zero: columns for $k=17,18,19$ are
not shown but contain only zero entries.  The entries marked with an asterix do not follow
from our theorems, but required additional arguments which we describe below.

\medskip

\begin{table}[ht!]
\begin{center}
\begin{tabular}{|c|c|c|c|c|c|c|c|c|c|c|c|c|c|c|c|}
\hline
\diagbox{$d$}{$k$} &
 2 & 3 & 4 & 5 & 6 & 7 & 8 & 9 & 10 & 11 & 12 & 13 & 14 & 15 & 16 \\
\hline
$3$ & \cellcolor{gray!15} 1 & 0 & 0 & 0 & 0 & 0 & 0 & 0 & 0 & 0 & 0 & 0 & 0 & 0 & 0  \\
\hline
$4$ & 1 & \cellcolor{gray! 15} 1 & 0 & 0 & 0 & 0 & 0 & 0 & 0 & 0 & 0 & 0 & 0 & 0 & 0  \\
\hline
$5$ & 1 & \cellcolor{gray!15} 1 & 0 & 0 & 0 & 0 & 0 & 0 & 0 & 0 & 0 & 0 & 0 & 0 & 0  \\
\hline
$6$ & 1 & 1 & \cellcolor{gray!15} 1 & 0 & 0 & 0 & 0 & 0 & 0 & 0 & 0 & 0 & 0 & 0 & 0  \\
\hline
$7$ & 1 & 1 & 1 & \cellcolor{gray!15} 1 & 0 & 0 & 0 & 0 & 0 & 0 & 0 & 0 & 0 & 0 & 0  \\
\hline
$8$ & 1 & 1 & 1 & \cellcolor{gray!15} 1 & 0 & 0 & 0 & 0 & 0 & 0 & 0 & 0 & 0 & 0 & 0  \\
\hline
$9$ & 1 & 1 & {\large \bf ?} & 1 & \cellcolor{gray!15} 1 & 0 & 0 & 0 & 0 & 0 & 0 & 0 & 0 & 0 & 0 \\
\hline
$10$ & 1 & 1 & {\large \bf ?} & 1 & \cellcolor{gray!15} 1 & {\bf $0^*$} & 0 & 0 & 0 & 0 & 0 & 0 & 0 & 0 & 0  \\
\hline
$11$ & 1 & 1 & {\large \bf ?} & 1 & 1 & \cellcolor{gray!15} 1 & 0 & 0 & 0 & 0 & 0 & 0 & 0 & 0 & 0  \\
\hline
$12$ & 1 & 1 & {\bf $1^*$} & 1 & 1 & \cellcolor{gray!15} 1 & {\large \bf ?} & 0 & 0 & 0 & 0 & 0 & 0 & 0 & 0  \\
\hline
$13$ & 1 & 1 & {\bf $1^*$} & 1 & {\large \bf ?} & 1 & \cellcolor{gray!15} 1 & 0 & 0 & 0 & 0 & 0 & 0 & 0 & 0  \\
\hline
$14$ & 1 & 1 & {\bf $1^*$} & 1 & {\large \bf ?} & 1 & \cellcolor{gray!15} 1 & {\large \bf ?} & 0 & 0 & 0 & 0 & 0 & 0 & 0  \\
\hline
$15$ & 1 & 1 & {\bf $1^*$} & 1 & {\large \bf ?} & 1 & 1 & \cellcolor{gray!15} 1 & 0 & 0 & 0 & 0 & 0 & 0 & 0  \\
\hline
$16$ & 1 & 1 & 1 & 1 & {\large \bf ?} & 1 & 1 & 1 & \cellcolor{gray!15} 1 & 0 & 0 & 0 & 0 & 0 & 0  \\
\hline
$17$ & 1 & 1 & {\large \bf ?} & 1 & {\large \bf ?} & 1 & {\large \bf ?} & 1 & \cellcolor{gray!15} 1 & 0 & 0 & 0 & 0 & 0 & 0  \\
\hline
$18$ & 1 & 1 & {\bf $1^*$} & 1 & {\bf $1^*$} & 1 & {\large \bf ?} & 1 & 1 & \cellcolor{gray!15} 1 & 0 & 0 & 0 & 0 & 0  \\
\hline
$19$ & 1 & 1 & {\bf $1^*$} & 1 & {\bf $1^*$} & 1 & {\large \bf ?} & 1 & 1 & \cellcolor{gray!15} 1 & 0 & 0 & 0 & 0 & 0  \\
\hline
$20$ & 1 & 1 & {\bf $1^*$} & 1 & {\bf $1^*$} & 1 & {\large \bf ?} & 1 & 1 & 1 & \cellcolor{gray!15} 1 & 0 & 0 & 0 & 0  \\
\hline
\end{tabular}
\caption{Limit of $\Pr(\mathcal{G}_{n,d} \,\, \text{has a $k$-star decomposition})$
for $3\leq d\leq 20$, where known.}
\label{table:big}
\end{center}
\end{table}

Since $\kSSCM{10}=6$ and $\kind{10}=7$, Theorem~\ref{thm:independence} and
Theorem~\ref{thm:main2} do not allow us to solve the case $(d,k)=(10,7)$.
To fill this cell in Table~\ref{table:big} we used the 1-RSB prediction $\alpha^{\text{1-RSB}}(10)\approx 0.281$ as
an upper bound on $\alpha^*(10)$,  together with (\ref{eq:indset-nonexistence}).

\medskip

Now we discuss the positive entries which have been marked with an asterix.
Let $(\mathcal{A}_n)$ and $(\mathcal{B}_n)$ be sequences of probability spaces 
on the same sequence of finite underlying sets $(\Omega_n)$. We say that
$(\mathcal{A}_n)$ and $(\mathcal{B}_n)$ are \emph{contiguous},
and write $\mathcal{A}_n\approx\mathcal{B}_n$, if for any
sequence of events $(\mathcal{E}_n)$ we have
\[ \lim_{n\to\infty} \Pr_{\mathcal{A}_n}(\mathcal{E}_n) = 1 \quad 
    \text{ if and only if } \quad
   \lim_{n\to\infty} \Pr_{\mathcal{B}_n}(\mathcal{E}_n) = 1. \]
The \emph{superposition} (also called \emph{graph-restricted sum} in~\cite{Wormald-survey}) of two random graph models $G_n$ and $H_n$
on the vertex set $[n]$, denoted $G_n\oplus H_n$,
is the probability space obtained by randomly choosing $G$ from $G_n$ 
and $H$ from
$H_n$ repeatedly until $E(G)\cap E(H)=\emptyset$, then 
returning $([n],E(G)\cup E(H))$. Using 
contiguity arithmetic~\cite[Corollary~4.17]{Wormald-survey} we know that 
if $d\geq 3$ and 
$a_1,\ldots, a_s$ are positive integers such that $d = \sum_{i=1}^s a_i$ then
\begin{equation}
\label{eq:contiguity} \mathcal{G}_{n,d} \approx \mathcal{G}_{n,a_1} \oplus  \mathcal{G}_{n,a_2}
   \oplus \cdots \oplus \mathcal{G}_{n,a_s}.
   \end{equation}
If two graphs with disjoint
edge sets have $k$-star decompositions, then their union also has a $k$-star
decomposition.  We used this fact, combined with (\ref{eq:contiguity}), 
to fill in some entries in Table~\ref{table:big} with $k\leq d/2$ which were
not covered by Theorem~\ref{thm:small-k}.  For example, 
$\mathcal{G}_{n,6}$ a.a.s.\ has a 4-star decomposition, 
and 
$\mathcal{G}_{n,12}\approx \mathcal{G}_{n,6}\oplus \mathcal{G}_{n,6}$,
which implies that $\mathcal{G}_{n,12}$ a.a.s.\ has a 4-star decomposition.
Similarly $\mathcal{G}_{n,20}$ a.a.s.\ has a 4-star decomposition since 
$\mathcal{G}_{n,20}\approx \mathcal{G}_{n,7}\oplus \mathcal{G}_{n,7}\oplus \mathcal{G}_{n,6}$. 

There is one caveat for the positive results obtained using contiguity arithmetic: a stronger divisibility condition may be needed for the existence of the $k$-star decompositions in the factors. For example, the divisibility condition for $(d,k)=(20,4)$ is that 8 divides $20n$,
(equivalently, that 2 divides $n$),
while for $(7,4)$ we require that 8 divides $7n$ (equivalently, that 8 divides $n$).
In other words, the cells marked $1^*$ in Table~\ref{table:big} are a.a.s.\ existence
results that hold as $n$ tends to infinity along an \emph{infinite subsequence} of $\mathcal{N}_{d,k}$.	

\medskip

To complete this section we comment on the unknown values in 
Table~\ref{table:big}.
As mentioned earlier, we believe that the a.a.s.\ existence for remaining unknown 
pairs with $k\leq d/2$
can be resolved by extending the result of~\cite{LTWZ} to even values of $k$.
In particular, Lov{\' a}sz et al.~\cite[Section~5.2]{LTWZ} claim (without proof)
that their
approach can be extended to $(3k-2)$-edge-connected graphs for even $k$,
which would provide an existence result for the cases 
\[ (d,k)\in \{ (10,4)\, (11,4),\, (17,4),\, (16,6),\, (17,6)\}.\]

We believe that the $(d,k)=(14,9)$ cell in Table~\ref{table:big} should be~1. 
Condition (\ref{eq:P1}) ensures that a certain function has a unique
local maximum at $x=1$.  In fact, our approach would still work if other 
local maxima exist but have smaller values than the value at $x=1$.
The function is difficult to analyse, which is why we restrict our 
attention to the simpler situation when there is only one local maximum.
This more complicated analysis will work for $(d,k)=(14,9)$ but we do not
include those details here. If, as we believe, $\mathcal{G}_{n,14}$ a.a.s.\ has a 9-star
decomposition, then $\alpha^*(14) \geq 2/9 = 0.222\ldots$,
by~(\ref{eq:indset-existence}).
This would improve on the best-known lower bound $0.2143$ proved by
Wormald~\cite[Theorem~5]{wormald95}. 
On the other hand, we do not offer any prediction for the cell $(12,8)$.

\subsection{Outline of the paper}\label{s:structure}

In Section~\ref{s:small-k} we prove Theorem~\ref{thm:small-k}.
Part (i) follows easily from known results about 2-star decompositions
and Eulerian orientations,
while parts (ii) and (iii) are proved using results of
Lov{\' a}sz, Thomassen, Wu and Zhang~\cite{LTWZ} and 
Thomassen~\cite{thomassen}, respectively.

In Section~\ref{s:negative} we prove Theorem~\ref{thm:independence}
and state the small subgraph conditioning theorem from Janson~\cite[Theorem~1]{janson}. The rest of the paper is devoted to verifying the assumptions
of this theorem applied to our problem.  In particular, Theorems~\ref{thm:main2} and Theorem~\ref{thm:main} are proved in Section~\ref{ss:consolidation}, 
utilising calculations performed in Sections~\ref{s:first-cycles} and~\ref{s:second-moment}. The most technical part of the second moment calculations are postponed until Section~\ref{s:maximisation}.  

Our small subgraph conditioning proof follows the same outline as that
given in~\cite{DP} for the case $(d,k)=(4,3)$, 
but it is noticeably more general. 
Carrying the proof with $(d,k)$ instead of two fixed numbers 
adds to the difficulty and technicality of the proofs. 
We have also provided
detailed proofs for the positive results which follow from~\cite{LTWZ,thomassen}
and use contiguity arithmetic to establish further positive results, 
as explained in Section~\ref{ss:small-values}.

\section{Small values of $k$}\label{s:small-k}

In this section we prove Theorem~\ref{thm:small-k}.
Given a $k$-star decomposition of a graph $G$, we can orient all edges of $G$ towards
the centre of the star containing that edge. This allows us to make a connection
with the theory of $\beta$-orientations, as described below.
We will repeatedly use the following result from Wormald~\cite{Wormald-d-connected}:
\begin{equation}
\label{eq:d-connected}
\text{ If $d\geq 3$ is fixed, then }\, \mathcal{G}_{n,d}\, \text{ is a.a.s.\ $d$-connected}.
\end{equation}
It is well known that a connected graph with an even number of edges
has a 2-star decomposition; see for example~\cite{kotzig} or \cite[Theorem~1]{CS}.
Next, let $G$ be a $2rk$-regular graph $G$ with $r\in\mathbb{Z}^+$.
Then $G$ has an Eulerian orientation, where 
each vertex $v$ has exactly $rk$ in-edges. 
These in-edges can be partitioned to give exactly $r$ distinct $k$-stars centred at 
$v$. These observations imply that Theorem~\ref{thm:small-k}(i) holds. 

\bigskip

Theorem~\ref{thm:small-k}(ii) follows immediately from
Thomassen~\cite[Theorem~5]{thomassen} using (\ref{eq:d-connected}).

To prove Theorem~\ref{thm:small-k}(iii) we use machinery from
Lov{\' a}sz et al.~\cite[Definition~1.9]{LTWZ}. 
Denote the set of integers modulo $k$ by $\mathbb{Z}_k$.  A function
$\beta:V(G)\to \mathbb{Z}_k$ is called a $\mathbb{Z}_k$-\emph{boundary} of $G$
if $\sum_{v\in V(G)}\beta(v)\equiv 0 \pmod{k}$.
	Given a $\mathbb{Z}_k$-boundary of $G$, an orientation $D$
	of $G$ is called a $\beta$-\emph{orientation} if
	for every vertex $v\in V(G)$,
\[ d^+_D(v) - d^-_D(v) \equiv \beta(v)\pmod{k}.\]
The following lemma establishes a connection with $k$-star decompositions.

\begin{lemma}
Let $G$ be a $d$-regular graph with $n$ vertices such that $2k$ divides $dn$,
where $k\geq 3$ is an odd positive integer. Then $G$ has a $k$-star
decomposition if and only if $G$ has a $\beta$-orientation with $\beta(v)\equiv d\pmod{k}$ for all $v\in V(G)$.
\label{lem:connection}
\end{lemma}

\begin{proof}
Since $k$ is odd and $d^+_D(v) + d^-_D(v) = d$ for all $v\in V(G)$, the condition $d^+_D(v) - d^-_D(v)\equiv d\pmod{k}$
is equivalent to saying that $k$ divides the in-degree of $v$.
When this condition holds for $v\in V(G)$ we may partition incoming edges
at $v$ into $k$-stars centred at $v$, and vice-versa.
\end{proof} 

Combining (\ref{eq:d-connected}) and the above lemma with the result 
of Lov{\' a}sz et al.~stated below, we
establish Theorem~\ref{thm:small-k}(iii) when $k \leq (d+1)/3$ and $k$ odd.

\begin{theorem}[{\cite[Theorem~1.12]{LTWZ}}]
Let $k\geq 3$ be an odd integer. Every $(3k-3)$-edge-connected graph
$G$ has a $\beta$-orientation for every $\mathbb{Z}_k$-boundary $\beta$
of $G$.
\end{theorem}

To cover the remaining range, that is $(d+1)/3 < k < d/2$,
we need a more technical result from~\cite{LTWZ}.
In this range, we let $\beta(v) = d -2k\in \{0,1,\ldots, k-1\}$ for all $v\in V(G)$.
Lov{\' a}sz et al.~\cite{LTWZ} use a function $\tau:V(G)\to \{0,\pm 1,\ldots, \pm k\}$
such that for each vertex $v\in V(G)$,
\[ \tau(v) \equiv \beta(v)\!\!\! \pmod{k} \qquad \text{ and }\qquad
   \tau(v) \equiv d \!\!\! \pmod{2}.\]
In our range we have $\tau(v)=\beta(v)=d-2k$ for all $v\in V$.
They extend $\tau$ to a function over subsets of $V(G)$ such that $|\tau(A)|\leq k$ for all $A\subseteq V(G)$.

We will state a special case of their theorem which suffices for our purposes,
obtained by applying~\cite[Theorem~3.1]{LTWZ}
to the disjoint union of a $d$-regular graph and an isolated vertex $z_0$,
using the fact that $\tau(v)$ is always nonzero and the observations above.
Let $e(A,\bar{A})$ denote the number of edges from $A$ to 
$\bar{A} = [n]\setminus A$, for all $A\subseteq [n]$.
A graph is said to be \emph{internally $r$-edge-connected} if for every
cut with at least two vertices on either side, the number of edges
across the cut is at least~$r$. 

\begin{theorem}[{\cite[Theorem~3.1]{LTWZ}}]
Let $G$ be a $d$-regular graph with $d\geq 3$, 
and let $k$ be an odd integer such that $(d+1)/3 < k <d/2$. 
If $G$ is internally $(3k-2)$-edge-connected,
then there exists a $\beta$-orientation $D$ of $G$, where 
$\beta$ is the $\mathbb{Z}_k$-boundary of $G$ defined by $\beta(v) = d-2k$ for all $v\in V(G)$.
\label{thm:LTWZ-restated}
\end{theorem}

The proof of Theorem~\ref{thm:small-k}(iii) is completed
by combining Lemma~\ref{lem:connection}, Theorem~\ref{thm:LTWZ-restated} and the following lemma.

\begin{lemma}
Let $d\geq 4$ be an integer. Then $\mathcal{G}_{n,d}$ is a.a.s.\ internally $2(d-1)$-edge-connected.
\end{lemma}

\begin{proof}
Let $A\subseteq [n]$ be a subset of vertices such that $2\leq |A| \leq n/2$. 
If $|A|=2$, then $e(A,\bar{A})\geq 2d-2$, with equality if
the two vertices in $A$ are adjacent. 

More generally, 
a.a.s.\ for any vertex subset $A$ with $3\leq |A| \leq 11(d-1)$, the induced
subgraph $\mathcal{G}_{n,d}[A]$ contains at most $|A|$ edges,
for example using Wormald~\cite[Lemma~2.7]{Wormald-survey}. 
	Hence a.a.s., every
	$A$ with $3\leq |A| \leq 11(d-1)$ satisfies
	\[ e(A,\bar{A}) \geq |A|(d-2) \geq 2(d-1),\]
	using the fact that $d\geq 4$.

Now suppose that $|A| > 11(d-1)$. 
It follows from 
Bollob{\' a}s~\cite[Theorem~1 and Corollary~2]{bollobas-iso}
that in $\mathcal{G}_{n,d}$, the following a.a.s.\ holds:  for all $A$ with $|A|  > 11(d-1)$ we have 
\[ e(A,\bar{A}) \geq \frac{2}{11} |A| > 2(d-1)\]
as required.
\end{proof}

\section{Configuration model}\label{s:negative}

Throughout the paper, we use $\log$ to denote the natural logarithm.
By convention, $0^0 = 1$ and $0\log 0=0$.  Let $(a)_b = a(a-1)\cdots(a-b+1)$
denote the falling factorial, for $a,b\in\mathbb{N}$.

Our calculations will be performed in the \emph{configuration model}
(or \emph{pairing model}) for $d$-regular graphs on $n$ vertices, denoted by
$\Omega_{n,d}$.  In the configuration model there are $n$ cells, each 
containing $d$ points.  A \emph{pairing}
is a partition of the $dn$ points into $dn/2$ unordered pairs.  Replacing
cells by vertices and pairs by edges, each pairing $P$ corresponds to
a multigraph $G(P)$ which may have loops or multiple edges.  The number of
pairings on $2a$ points is denoted by
\begin{equation}
\label{nr-pairings}
 M(2a) = \frac{(2a)!}{a!\, 2^a}.
\end{equation}
Bender and Canfield~\cite{BC} proved that the probability that a randomly chosen configuration
from $\Omega_{n,d}$ is simple is 
\begin{equation}\label{eq:simple}
	\Pr(\text{Simple}) \sim \exp\big(-(d^2-1)/4\big).
	\end{equation}
Let $Y$ be any random variable on $\Omega_{n,d}$ 
which satisfies $Y(P)=Y(P')$ for any $P,P'\in\Omega_{n,d}$ such that $G(P)=G(P')$.
The corresponding random variable on graphs, $Y_{\mathcal{G}}$, is 
defined by $Y_{\mathcal{G}}\big(G(P)\big) = Y(P)$.
Since every $d$-regular graph on $[n]$ corresponds to the same
number of pairings (specfically, $(d!)^n$ pairings).
it follows from (\ref{eq:simple}) that
\begin{equation}
\label{negative-useful}
\Pr(Y_{\mathcal{G}} = 0) = \Pr(Y = 0 \mid \text{Simple})
   \leq \frac{\Pr(Y = 0)}{\Pr(\text{Simple})} = O\big(\Pr(Y = 0)\big).
   \end{equation}
   In particular, if $\Pr(Y=0)\to 0$, then a.a.s.\ $Y_{\mathcal{G}}>0$.

\subsection{Proof of Theorem~\ref{thm:independence}}\label{sec:ind-negative}

Our proof of Theorem~\ref{thm:independence} relies on the first moment approach.  Similar calculations can be found in~\cite{DSS,FL}.
Let $Z_\alpha$ denote the number of independent sets of size $\alpha n$ in the 
configuration model $\Omega_{n,d}$, where $\alpha\in (0,1)$ is fixed.  Then
\[ \E Z_\alpha = \binom{n}{\alpha n}\, (dn-d\alpha n)_{d\alpha n}\, \frac{M(dn-2d\alpha n)}{M(dn)}.
\]
This follows as there are $\binom{n}{\alpha n}$ ways to select a set $U$ of
$\alpha n$ cells, then we wish to pair all $d\alpha n$ points in these cells to
points in cells outside $U$, which can be done in $(dn-d\alpha n)_{d\alpha n}$ ways.
We can then complete these $d\alpha n$ pairs to a full pairing in $M(dn-2d\alpha n)$
ways. In this pairing, $U$ corresponds to an independent set.  Finally,
we multiply by $1/M(dn)$ which is the probability that we observe this
pairing under the uniform distribution.

Applying (\ref{nr-pairings}) and Stirling's approximation gives
\begin{align}
	\E Z_\alpha &= \frac{n!\, (dn-d\alpha n)!\, (dn/2)!\, 2^{d\alpha n}}{(\alpha n)!\, \big((1-\alpha)n\big)!\, (dn/2-d\alpha n)!\, (dn!)} \nonumber \\
  &\sim \sqrt{\frac{1}{2\pi\alpha(1-2\alpha)\, n}}\,
	\left(\frac{(1-\alpha)^{(d-1)(1-\alpha)}}{\alpha^\alpha\, (1-2\alpha)^{d(1-2\alpha)/2}}\right)^n. \label{eq:EZ}
\end{align}
Finally we substitute $\alpha=(2k-d)/(2k)$, writing $Z$ instead of
$Z_{(2k-d)/(2k)}$ for ease of notation, giving 
\[ \E Z \sim \frac{k}{\sqrt{\pi (2k-d)\, (d-k)\,n}}\, 
  \left(\frac{d^{d-1}}{(2k-d)^{(2k-d)/d}\, (d-k)^{d-k} \, 2^{(d^2-2k)/d}\, k^{k(d-2)/d}}
   \right)^{dn/(2k)}.\]
If $k > \kind{d}$, then the base of the exponential factor is bounded above by~1, and hence $\E Z \to 0$.  
Hence the expected number of 
independent sets in $\mathcal{G}_{n,d}$ of size $n-dn/(2k)$ also
tends to zero, by applying (\ref{negative-useful}) with $Y=Z$.  
This proves the result in this case, using (\ref{eq:indset-nonexistence})
and Markov's inequality.  We defer the proof of the result for $(d,k)=(5,4)$ to Lemma~\ref{L:k-star}.
 
\subsection{Small subgraph conditioning method}\label{s:SSCM}

To prove Theorem~\ref{thm:main} and Theorem~\ref{thm:main2} we will apply the small subgraph conditioning method, introduced by Robinson and Wormald~\cite{RW92}. 
This statement is taken from \cite{janson}.
	
	\begin{theorem}[Janson \protect{\cite[Theorem 1]{janson}}] \label{thm:subgraph}
		Let $\lambda_j > 0$ and $\delta_j \ge -1$, $j = 1, 2, \dots$, be constants and suppose that for each $n$ there are random variables $X_{j,n}$, $j = 1,2, \dots$, and $Y_n$ (defined on the same probability space) such that $X_{j,n}$ is nonnegative integer valued and $\E Y_n \ne 0$ (at least for large $n$), and furthermore the following conditions are satisfied:
		\begin{enumerate}
			\item[\emph{(A1)}] $X_{j,n} \overset{d}{\longrightarrow} Z_j$ as $n \to \infty$ jointly for all $j$, where $Z_j \sim \operatorname{Po}(\lambda_j)$ are independent Poisson random variables;
			\item[\emph{(A2)}] For any finite sequence $x_1, \dots, x_m$ of nonnegative integers,
			\[
			\frac{\E(Y_n | X_{1,n} = x_1, \dots, X_{m,n} = x_m)}{\E Y_n} \to \prod_{j=1}^m (1+ \delta_j)^{x_j} e^{- \lambda_j \delta_j} \quad \text{ as } n \to \infty;
			\]
			\item[\emph{(A3)}] $\displaystyle \sum_{j \ge 1} \lambda_j \delta_j^2 < \infty$;
			\item[\emph{(A4)}] $\dfrac{\E Y_n^2}{(\E Y_n)^2} \to \exp\left( \displaystyle \sum_{j \ge 1} \lambda_j \delta_j^2 \right) \quad \text{ as } n \to \infty$.
		\end{enumerate}
		Then
		\[
		\frac{Y_n}{\E Y_n} \,\,\, \overset{d}{\longrightarrow} \,\,\, W = \prod_{j=1}^\infty (1+\delta_j)^{Z_j} e^{-\lambda_j \delta_j} \quad \text{ as } n \to \infty;
		\]
		moreover, this and the convergence in \emph{(A1)} hold jointly. 
	In particular, $W > 0$ almost surely if and only if every $\delta_j > -1$. 
	\end{theorem}
	
Janson remarks in \cite{janson} that the index set $\mathbb{Z}^+$ may be replaced by any other countably infinite set. We mention this since all of our asymptotics
are performed as $n\to\infty$ along the countably infinite set $\mathcal{N}_{d,k}$.

	In order to verify condition (A2), the following lemma is helpful.

	\begin{lemma}[\protect{Janson \cite[Lemma 1]{janson}}]
	Let $\lambda_j' \ge 0$, $j = 1, 2, \dots$ be constants. Suppose that \emph{(A1)} holds, that $Y_n \ge 0$ and that
	\[
	\text{\emph{(A2$'$)}} \qquad \frac{\E(Y_n (X_{1,n})_{x_1} \cdots (X_{m,n})_{x_m})}{\E Y_n}  \to \prod_{j=1}^m (\lambda_j' )^{x_j} \quad \text{ as } n \to \infty,
	\]
	for every finite sequence $x_1, \dots, x_m$ of nonnegative integers. Then condition \emph{(A2)} holds with $\delta_j = \frac{\lambda_j'}{\lambda_j}-1$.
	\end{lemma}

The small subgraph conditioning method has been applied to prove several
structural results about random regular graphs. See for example the results surveyed in~\cite{janson,Wormald-survey}.  For those new to the method, a good place to start may be the short and
accessible work of Pra{\l}at and Wormald~\cite{pralat-wormald}, who used small subgraph conditioning
to prove that almost all 5-regular graphs contain a 3-flow.  
Intuition behind the method, focussing in particular on why conditioning on 
numbers of short cycles can help to control the variance, is provided in~\cite[Section~2]{DP}.

\subsection{Applying the method to our problem}\label{s:applying}

We now define the random variables $Y, X_1,X_2,\ldots $ that we will use in our
application of the small subgraph conditioning method. (We suppress the
dependence on $n$, for ease of notation.)
Recall that when $k > d/2$, a $k$-star decomposition can be viewed as an 
orientation of the edges of the graph such that every vertex has in-degree 0 or $k$.
We say a vertex is a \emph{centre} if it has in-degree $k$,
otherwise it is a \emph{leaf}.   
Let $Y=Y_n$ be the number of ways to orient all pairs in a uniformly random
pairing from $\Omega_{n,d}$, such that every cell has in-degree $k$ or 0.
This corresponds to an edge-disjoint decomposition of the
(multigraph corresponding to the) pairing into copies  of $k$-stars.
For $j\geq 1$ let $X_j=X_{j,n}$ be the number of $j$-cycles in a uniformly random
pairing from $\Omega_{n,d}$.
Bollob{\' a}s~\cite{bollobas1980probabilistic} 
proved that $X_j \to Z_j$ as $n \to \infty$, where $Z_j$ are asymptotically independent Poisson random variables with mean
	\begin{equation} \label{eq:lambda}
	\lambda_j = \frac{(d-1)^j}{2j}.
	\end{equation}
The remainder of the paper is devoted to proving Theorem~\ref{thm:main}
and Theorem~\ref{thm:main2}, using the small subgraph conditioning method
(Theorem~\ref{thm:subgraph})
applied to the random variables $Y$, $X_1, X_2, X_3,\ldots$, 
where $n$ tends to
infinity along the index set $\mathcal{N}_{d,k}$.

The first moment calculations are given at the start of 
Section~\ref{s:first-cycles}, and the effect of short cycles is investigated in 
Section~\ref{s:short-cycles}.  
The second moment calculations are presented in Section~\ref{s:second-moment}.  These are the most technical calculations in the paper, and rely on a certain function having
a unique maximum in the interior of a given domain: this fact is proved
in Section~\ref{s:maximisation}.

Successful application of the small subgraph conditioning method will
allow us to conclude that a.a.s.\ $Y>0$.
Now let $Y_{\mathcal{G}} = Y_{\mathcal{G},d,k}$ be the number of
$k$-star decompositions of $\mathcal{G}_{n,d}$.  
Applying (\ref{negative-useful}) to $Y$ and $Y_{\mathcal{G}}$
we conclude that a.a.s.\ $Y_{\mathcal{G}}>0$.

\section{First moment and the effect of short cycles}\label{s:first-cycles}

With notation as above,
\begin{equation}
	\label{first-moment}
	\E Y =  \binom{n}{\frac{dn}{2k}}\, \binom{d}{k}^{dn/(2k)}\, (dn/2)!\,\, \frac{1}{M(dn)}.
\end{equation}
Here the first factor chooses the centres, the next factor identifies
the in-points in each centre, the third factor fixes a pairing which
pairs each in-point with an out-point, and the final factor is the 
probability that we observe this particular pairing.  

Applying Stirling's formula, we find that
\begin{equation}
	\label{expectation}
	\E Y  \sim \frac{k}{\sqrt{2k-d}}\, \left(\frac{\binom{d}{k}\, k^{2k/d}}{2^{k(d-2)/d}\, d\, (2k-d)^{(2k-d)/d}}
	\right)^{dn/(2k)}.
\end{equation}
Next we prove the following result, which will be needed in Section~\ref{s:second-moment}.

\begin{lemma}\label{L:k-star}
	Suppose that $d\geq 4$ and $d/2 < k \leq \kind{d}$. 
	If $(d,k)\neq (5,4)$, then 
	\begin{equation}\label{orientation-expectation-threshold}
	\frac{\binom{d}{k}\, k^{2k/d}}{2^{k(d-2)/d}\, d\, (2k-d)^{(2k-d)/d}}
	>1\end{equation}
	and hence $\E Y \rightarrow\infty$.  If $(d,k)=(5,4)$, then $\E Y\rightarrow 0$.
\end{lemma}

\begin{proof}
	Recall the asymptotic formula for $\E Z_\alpha$ from (\ref{eq:EZ}),
	where $Z_\alpha$ is the number of independent sets of 
	size $\alpha n$ in $\Omega_{n,d}$.
For $0\leq \alpha < 1/2$, let
	\[ h_d(\alpha) = (d-1)(1-\alpha)\log(1-\alpha) - \alpha \log\alpha - \frac{d}{2}\, (1-2\alpha)\log(1-2\alpha)
\]
be the logarithm of the exponential part of the asymptotic formula for $\E Z_\alpha$.
	As noted by Ding, Sly and Sun~\cite[Lemma~2.1]{DSS}, $h_d(\alpha)$ is concave on $[0,1/2)$, and $h_d'(\alpha)>0$ for small values of $\alpha$. Since $h_d(0)=0$ and $h_d(1/2)<0$, this implies that there is a unique value of $\alpha\in (0,1/2)$ such that $h_d(\alpha)=0$.  The inequality (\ref{eq:k+}) holds if and only if
	$h_d(\alpha^*) > 0$ holds for $\alpha^* = (2k-d)/(2k)$.
	Since $(2k-d)/(2k) = 1-d/(2k)$ is an increasing function of $k$, it follows that 
\begin{equation}
\label{eq:monotonicity}
\text{(\ref{eq:k+}) holds for all }\, d/2 <k\leq \kind{d}.
\end{equation}

Since (\ref{eq:monotonicity}) holds, to establish (\ref{orientation-expectation-threshold}) it is sufficient to prove that
\begin{align*}
1 < \frac{\binom{d}{k}\, k^{2k/d}}{2^{k(d-2)/d}\, d}\, \cdot
  \frac{2^{(d^2-2k)/d}\, k^{k(d-2)/d}\, (d-k)^{d-k}}{d^{d-1}} 
   = \frac{\binom{d}{k}\, k^k\, (d-k)^{d-k}\, 2^{d-k}}{d^d}.
  \end{align*}
Rewrite this inequality as
\begin{equation}
\label{eq:dagger}
\frac{\binom{d}{k}\, k^k\, (d-k)^{d-k}}{2^k} > \left(\frac{d}{2}\right)^d.
\end{equation}
We start by dealing with the cases $d\in \{4,5,6,7,8,9\}$ computationally.
Let $c(d,k)$ denote the left hand side of (\ref{orientation-expectation-threshold}). 
Recall that $\kind{4}=3$, $\kind{5} = 4$, $\kind{6} = 4$, $\kind{7}=5$, $\kind{8}=5$,
$\kind{9}=6$.
Computing $c(d,k)$
for relevant values to 3 decimal figures gives:
\begin{align*}
c(4,3) &= 1.299, \quad c(5,3) = 2.146,\quad c(5,4) = 0.901,\quad c(6,4) = 1.984,\\
 c(7,4) &= 3.365, \quad c(7,5) = 1.571, \quad c(8,5) = 3.271,\quad
  c(9,5) = 5.651,\quad c(9,6) = 2.778. 
 \end{align*}
This shows that (\ref{orientation-expectation-threshold}) holds
when $d\in \{4,5,6,7,8,9\}$ and $d/2 < k \leq \kind{d}$, except for the special case
$(d,k) = (5,4)$ where $c(5,4) < 1$ and hence $\E Y\rightarrow 0$.

For the remainder of the proof, suppose that $d\geq 10$.
We claim that the left hand side of (\ref{eq:dagger}) is monotonically decreasing
for $k\leq d-2$.  This is true if
	\[ \frac{\binom{d}{k+1}\, (k+1)^{k+1}\, (d-k-1)^{d-k-1}}{\binom{d}{k}\,
	k^k\, (d-k)^{d-k}}  = \frac{(1 + 1/k)^k}{\big(1 + 1/(d-k-1)\big)^{d-k-1}}< 2,\]
	and the above inequality holds as
\[ \left(1 + \frac{1}{k}\right)^k \leq e < 4 \leq 2\left(1 + \frac{1}{d-k-1}\right)^{d-k-1}\]
when $d/2 < k \leq d-2$.  

Next, when $d\geq 10$ we have
$36\big(1 - 3/d\big)^d > 1$
and hence
\[ 27\, \binom{d}{3}\, (d-3)^{d-3} > 2^{d-3}\, \left(\frac{d}{2}\right)^d.\]
This implies that (\ref{eq:dagger}) holds at $k=d-3$ when $d\geq 10$.
In other words, (\ref{eq:k+}) implies
(\ref{orientation-expectation-threshold}) when $k=d-3$ and $d \geq 10$.
By monotonicity of the left hand side of (\ref{eq:dagger}), proved above,
it follows that (\ref{eq:k+}) implies
(\ref{orientation-expectation-threshold}) for all $d\geq 10$ and $k$
such that $d/2 < k \leq d-3$.

Next,  we show that (\ref{eq:k+})
fails when $k=2d/3$, and hence $\kind{d} < 2d/3$.
Note that $2d/3 \leq d-3$ as $d\geq 10$.
After substitution and cancellation, we see that 
(\ref{eq:k+}) fails when $k=2d/3$ if and only if
\[ \left(\frac{3}{2^{5/3}}\right)^d \leq d,\]
and since $3 < 2^{5/3}$,  this inequality holds for all $d\geq 1$.

Putting this together, we see that if $d\geq 10$ and $k\leq \kind{d}$,
then (\ref{eq:k+}) holds, and hence (\ref{orientation-expectation-threshold}) 
holds. This completes the proof.
\end{proof}

\subsection{Effect of short cycles} \label{s:short-cycles}

A subpairing is a set of at most $dn/2$ pairs of points.
If the corresponding multigraph $G(P)$ is a cycle of length $j$, 
then we say that the subpairing $P$ is a cycle of length $j$.

Let $\xvec = (x_1,\ldots, x_m)$ be a sequence of nonnegative integers,
for some fixed $m\geq 1$.  Define $J = \sum_{j=1}^m j x_j$ and assume
that $J > 0$.  
Write $r=\sum_{j=1}^m x_j$. 

Let $\mathcal{S}(\xvec)$ be the set of sequences $(P_1,\ldots, P_r)$ 
of subpairings such that
$G(P_1),\ldots$, $G(P_r)$ are distinct cycles so that the first $x_1$ have length 1, the
next $x_2$ have length 2, etc, and the last $x_m$ cycles have length $m$.
Next, let $\mathcal{S}^\ast(\xvec)$ be the set of all sequences $(P_1,\ldots, P_r)$ in
$\mathcal{S}(\xvec)$ so that the cycles $G(P_1),\ldots, G(P_r)$ are vertex-disjoint.
Then
\begin{equation}
\label{eq:cycle-sum}
 \frac{\E( Y(X_1)_{x_1}\cdots (X_m)_{x_m} )}{\E Y} =
 \frac{1}{\E Y}\, 
   \sum_{(P_1,\ldots, P_r)\in\mathcal{S}(\xvec)}\, 
    \sum_{P\supset P_1\cup \cdots \cup P_r} \, \frac{Y(P)}{M(dn)}
\end{equation}
where the second sum is over all pairings $P\in\Omega_{n,d}$ which contain $P_1\cup \cdots \cup P_r$.
First we perform the summation over $(P_1,\ldots, P_r)\in\mathcal{S}^\ast(\xvec)$,
where all cycles are vertex-disjoint.  Later we prove that the full sum
is asymptotically equal to this restricted sum.
For all positive integers $j$, define
\begin{equation}
\label{eq:delta}
   \delta_j = \left(\frac{d-2k+1}{d-1}\right)^j.
\end{equation}
Observe that $\delta_j > -1$ for all $j\geq 1$ whenever $k<d$.

\medskip

\begin{lemma}
\label{lem:disjoint-cycles}
Let $\xvec = (x_1,\ldots, x_m)$ be a sequence of nonnegative integers,
for some fixed $m\geq 1$, and define $\Sigma^\ast(\xvec)$ by
\[ 
\Sigma^\ast(\xvec) = \sum_{(P_1,\ldots, P_r)\in\mathcal{S}^\ast(\xvec)}
    \sum_{P\supset P_1\cup \cdots \cup P_r}\,  \frac{Y(P)}{M(dn)}.\]
Then
\[ \frac{\Sigma^\ast(\xvec)}{\E Y} 
   \sim  \prod_{j=1}^m \big( \lambda_j (1+\delta_j)\big)^{x_j}.
\]
\end{lemma}

\begin{proof}
We will choose distinct vertex-disjoint cycles $C_1,\ldots, C_r$ so that
$C_i$ has length $j_i$, where $(j_1,\ldots, j_r)$ is the vector 
\[ \big( \underbrace{1,1,\ldots, 1}_{x_1}, \underbrace{2,2,\ldots, 2}_{x_2},\ldots, 
    \underbrace{m,m\ldots, m}_{x_m} \big)\]
of lengths.  Note that $\sum_{i=1}^r j_i = J$.
There are
\[ \frac{(n)_J}{2^r\, \prod_{i=1}^r  j_i}\] 
ways to choose these cycles.  Next, we choose an orientation of the edges
of cycles $C_1,\ldots, C_r$ in
\[  2^r \prod_{i=1}^r \sum_{p_i=0}^{\lfloor j_i/2\rfloor} \binom{j_i}{2p_i}\]
ways, where $p_i$ is the number of sources on $C_i$, which equals the number of
sinks on $C_i$.   This follows as sources and sinks must alternate around $C_i$,
so having chosen a subset of $2p_i$ vertices of $C_i$, there are 2 possible 
orientations of $C_i$ which make these $2p_i$ vertices the sources/sinks.
Then there are $( d(d-1) )^J$ ways to assign points
to the pairs corresponding to the $J$ edges in $C_1\cup\cdots \cup C_J$.
Given the above orientation, we already know which point is the out-point of each pair
of these $J$ pairs.

Next, we will choose which vertices are going to be centres and which
are going to be leaves.  
Suppose that $\ell_i$ sources on $C_i$ are leaves, so that $p_i-\ell_i$
sources on $C_i$ are centres.  The $j_i-p_i$ non-sources on $C_i$
must all be centres, as they have positive in-degree.  Therefore there are 
$j_i-\ell_i$ centres on $C_i$, giving $J-L$ centres in $C_1\cup\cdots\cup C_r$,
where $L = \sum_{i=1}^r \ell_i$.  There are
\[ \prod_{i=1}^r \binom{p_i}{\ell_i}\]
ways to identify all leaves, and hence all centres, on $C_1\cup\cdots\cup C_r$.

There are $\binom{n-J}{\frac{dn}{2k}-J+L}$ ways to choose the remaining centres,
from the $n-J$ vertices not in $C_1\cup \cdots\cup C_r$.   
Since $J$ and $L$ are bounded, we have
\begin{equation}
\binom{n-J}{\frac{dn}{2k}-J+L}\sim \bigg(\frac{d}{2k}\bigg)^{J-L}\, \bigg(\frac{2k-d}{2k}\bigg)^L\, \binom{n}{\frac{dn}{2k}}.
\label{eq:convenient}
\end{equation}
If $v$ is one of these new
centres, then there are $\binom{d}{k}$ ways to select the in-points for $v$.
If $v$ is one of the $p_i-\ell_i$ sources on $C_i$ which becomes a centre,
then there are $\binom{d-2}{k}$ ways to choose the in-points for $v$.
If $v$ is one of the $\ell_i$ sources on $C_i$ which become a leaf, then all
points of $v$ are out-points.
If a vertex $v$ of $C_i$ is neither a source nor a sink, then
there are $\binom{d-2}{k-1}$ ways to select the remaining in-points
for $v$.  Finally,  if $v$ is a sink on $C_i$, then there are
$\binom{d-2}{k-2}$ ways to select the remaining in-points for $v$. 
Now that all in-points are identified, there are $(dn/2-J)!$ ways to
complete the oriented pairing, matching out-points to in-points, since
the pairs in $C_1\cup \cdots\cup C_r$ are already chosen and oriented.    
Again we divide by $M(dn)$ for
the probability that we observe this pairing.

Using (\ref{eq:convenient}), this leads to the expression
\begin{align*}
& \Sigma^\ast(\xvec)\\ 
&\sim \frac{\big(d(d-1)\, n\big)^J}{\big(\prod_{i=1}^r j_i\big)}\, 
\frac{(dn/2-J)!}{M(dn)}\,
\binom{d}{k}^{dn/(2k)-J}
\binom{n}{\frac{dn}{2k}}\, \bigg(\frac{d}{2k}\bigg)^J\, 
\\ & \hspace*{2mm} {} \times
\prod_{i=1}^r\, 
\sum_{p_i=0}^{\lfloor j_i/2\rfloor} \binom{j_i}{2p_i}\,
\sum_{\ell_i=0}^{p_i}\, 
\bigg(\frac{2k-d}{2k} \, \binom{d}{k}\bigg)^{\ell_i}\, \
\binom{p_i}{\ell_i}\,
\binom{d-2}{k}^{p_i-\ell_i}\,
  \binom{d-2}{k-1}^{j_i-2p_i}\, \binom{d-2}{k-2}^{p_i}. 
\end{align*}
Next we divide by $\E(Y)$, using (\ref{first-moment}). Since $J$
is bounded, after much simplification this gives
\begin{align*}
 &\frac{\Sigma^\ast(\xvec)}{\E Y} \\
  &\sim  
\prod_{i=1}^r \, \frac{(d-k)^{j_i}}{j_i}\, 
\sum_{p_i=0}^{\lfloor j_i/2\rfloor}\,
   \binom{j_i}{2p_i}\,
  \left(\frac{(k-1)(d-k-1)}{k(d-k)}\right)^{p_i}\,
   \sum_{\ell_i=0}^{p_i} \binom{p_i}{\ell_i}\, \left( \frac{(2k-d)(d-1)}{(d-k)(d-k-1)}\right)^{\ell_i}\\
 &=  
  \prod_{i=1}^r\, 
\frac{(d-k)^{j_i}}{j_i}\, 
\sum_{p_i=0}^{\lfloor j_i/2\rfloor} \, \binom{j_i}{2p_i}\,
  \left(\frac{(k-1)(d-k-1)}{k(d-k)}\right)^{p_i}\, 
     \left(1 + \frac{(2k-d)d}{(d-k)(d-k-1)} \right)^{p_i}\\
 &=  \prod_{i=1}^r\, \frac{(d-k)^{j_i}}{j_i}\, \sum_{p_i=0}^{\lfloor j_i/2\rfloor} 
  \, \binom{j_i}{2p_i}
     \left( \frac{k-1}{d-k} \right)^{2p_i}.
\end{align*}
This shows that asymptotically, the effect of each (disjoint) cycle is independent.
The sum over only even values of $N$ in the
generating function $F(x) = \sum_{N} a_n x^N$ 
is given by $\nfrac{1}{2}\big(F(x) + F(-x)\big)$.
Using this, we conclude that 
\begin{align*}
\frac{\Sigma^\ast(\xvec)}{\E Y} & \sim 
   \prod_{i=1}^r\, \frac{(d-k)^{j_i}}{2{j_i}}\, \left(
    \left(1 + \frac{k-1}{d-k}\right)^{j_i} + \left(1 - \frac{k-1}{d-k}\right)^{j_i}
   \right)\\
  &= \prod_{i=1}^r\,
    \frac{(d-1)^{j_i}}{2j_i}\left(1 + \left(\frac{d-2k+1}{d-1}\right)^{j_i}\right).
\end{align*}
Since exactly $x_j$ values of $i$ satisfy $j_i=j$, for $i=1,\ldots, m$,
this completes the proof.
\end{proof}

It remains to prove that the sum over $\mathcal{S}^\ast(\xvec)$ dominates
the whole sum.

\begin{lemma}
Let $\xvec = (x_1,\ldots, x_m)$ be a sequence of nonnegative integers,
for some fixed $m\geq 1$.   Then, recalling \emph{(\ref{eq:lambda})} and 
\emph{(\ref{eq:delta})},
\[ 
 \frac{\E( Y(X_1)_{x_1}\cdots (X_m)_{x_m} ) }{\E Y}
   \sim  \prod_{j=1}^m \big( \lambda_j (1+\delta_j)\big)^{x_j}.
\]
\label{lem:all-cycles}
\end{lemma}

\begin{proof}
Using (\ref{eq:cycle-sum}), it suffices to prove that the sum over all
$(P_1,\ldots P_r)\in\mathcal{S}(\xvec) \setminus \mathcal{S}^\ast(\xvec)$
contributes negligibly. 
We adjust the argument given in Lemma~\ref{lem:disjoint-cycles}.
Suppose that we choose cycles $(C_1,\ldots, C_r)$ with the appropriate lengths
which are not vertex-disjoint.  If $C_1\cup \cdots \cup C_r$ has $a$ vertices
and $b$ edges, then there are $O(n^a)$ ways to choose and orient 
the cycles and choose points for the pairs corresponding to edges in the cycles.
There are $O(1)$ ways to choose centres within these cycles, and $O\Big(\binom{n}{\frac{dn}{2k}}\Big)$ ways to choose the remaining centres.  There are $O\Big(\binom{d}{k}^{dn/(2k)}\Big)$ ways to choose the remaining out-points for all vertices, and $(dn/2-b)!$
ways to complete the oriented pairing, matching out-points to in-points.
Note that $a<b$ as the cycles are not vertex-disjoint.
Using (\ref{expectation}), it follows that the contribution of all $(P_1,\ldots, P_r)$
such that the corresponding cycles together have $a$ vertices and $b$ edges,
and are not vertex-disjoint, is $O(n^{a-b})=o(1)$.  The proof follows since
there are only $O(1)$ choices for $a$ and $b$.
\end{proof}

\bigskip

To complete this section we investigate the infinite sum needed for (A3).

\begin{lemma}
\label{lem:magic-number}
Let $\lambda_j$, $\delta_j$ be defined in \emph{(\ref{eq:lambda}), (\ref{eq:delta})},
for all positive integers $j$.
If $k$ is an integer such that $5\leq d < 2k$ and $(2k-d)^2 <  4k- d-2$, then
\[  \sum_{j=1}^\infty \lambda_j \delta_j^2 
  = \dfrac{1}{2}\log\left(\frac{d-1}{4k-d-2-(2k-d)^2}\right).
\]
\end{lemma}

\begin{proof}
From the definition of $\lambda_j$ and $\delta_j$, we have
\[
\sum_{j=1}^\infty \lambda_j \delta_j^2 =
  \dfrac{1}{2} \sum_{j=1}^\infty \frac{(d+1-2k)^{2j}}{j\, (d-1)^j}.\]
Our assumptions on $k$ imply that $(d+1-2k)^2 < d-1$, so the series
converges.
Hence
\begin{align*}
  \sum_{j=1}^\infty \lambda_j \delta_j^2 
  &= -\dfrac{1}{2}\log\left(1 - \frac{(d+1-2k)^2}{d-1}\right)
  = \dfrac{1}{2}\log\left(\frac{d-1}{4k-d-2-(2k-d)^2}\right),
\end{align*}
as claimed.
\end{proof}

\section{Second moment} \label{s:second-moment}

Following~\cite{DP}, the \emph{signature} $S$ of an orientation $O$
is the set of the $dn/2$ in-points of the orientation. If the orientation
corresponds to a $k$-star decomposition, then 
every cell/vertex has either $k$ or~0 points in $S$.

Let $(S_1,S_2)$ be an ordered pair of signatures, corresponding
to an ordered pair of orientations $(O_1,O_2)$.
For $i=0,1,\ldots, d-k$, let $B_i$ be the number of
vertices $v$ which are centres in both 
orientations such that exactly $k-i$ points corresponding to $v$
belong to $S_1\cap S_2$.  Hence the total number of vertices
which are centres in both orientations is exactly 
$\sum_{i=0}^{d-k} B_i$.  The number of 
vertices $v$ such that $v$ is a centre in $O_1$ but not in $O_2$ is
$\dfrac{dn}{2k}-\sum_{i=0}^{d-k} B_i$, and the same number of
vertices are centres in $O_2$ but not in $O_1$.  Exactly
$n - \dfrac{dn}{k} + \sum_{i=0}^{d-k} B_i$ 
vertices are leaves in both $O_1$ and $O_2$. 

Therefore, the number of ways to ``classify'' all vertices
is
\[ \binom{n}{B_0,B_1,\ldots, B_{d-k}, \, \frac{dn}{2k} - \sum B_i,\, \frac{dn}{2k} - \sum B_i,\, n - \frac{dn}{k} + \sum B_i}.\]
(Here, and throughout this section, any sum involving $B_i$ is assumed to be over $i=0,\ldots, d-k$
unless otherwise specified.)

Next, we must choose $(S_1,S_2)$ to match these
parameters $(B_0,\ldots, B_{d-k})$.
If a vertex is a leaf in both orientation, then it contributes
no points to $S_1$ or $S_2$.
If $v$ is a leaf in $O_1$ and a centre in $O_2$, then there
are $\binom{d}{k}$ ways to assign in-points to $v$ in $(S_1,S_2)$.
The same is true if $v$ is a centre in $O_1$ and a leaf in $O_2$.

Now suppose that $v$ is a centre in both $O_1$ and $O_2$,
with $k-i$ points from $v$ in $S_1\cap S_2$.  There are
\begin{equation}
\x_i:= \binom{d}{k-i,\, i,\, i,\, d-k-i} = \frac{d!}{(k-i)!\, (d-k-i)!\, i!^2}
\label{eq:tau-def}
\end{equation}
ways to assign in-points, since $k-i$ points are in-points in both
$O_1$ and $O_2$ (meaning that that part has the same orientation
in both orientations), $d-k-i$ points are in-points in both $O_1$ and
$O_2$ (and again, that part has the same orientation in $O_1$ and $O_2$),
and there are $i$ points which are out-points in $O_1$ but
in-points in $O_2$: the same is true for points which are out-points
in $O_2$ and in-points in $O_1$.
Note that $\sum_{i=0}^{d-k} \x_i = \binom{d}{k}^2$.
Overall, there are
\[ \binom{d}{k}^{\frac{dn}{k} - 2\sum B_i}\, \,\,\,\prod_{i=0}^{d-k} 
  \x_i^{B_i}
\]
ways to choose $(S_1, S_2)$.

Finally, we must pair up all the points, which we can do in
\[
   \left( \sum (k-i) B_i \right)!\, \left( dn/2 - \sum (k-i) B_i\right)!
       \]
ways.  The first factor matches all points in $S_1\cap S_2$
to all points in the complement of $(S_1\cup S_2)$, while the
first factor matches all points in $S_1\setminus S_2$ to the
points in $S_2\setminus S_1$.  We  multiply by $1/M(dn)$ as usual,
to give the probability that this particular pairing is observed.

Define the domain
\[
  \mathcal{D} = \Big\{ \, \Bvec = (B_0,\ldots, B_{d-k}) \in\mathbb{Z}^{d-k+1}\,\, 
    \Big| \,\, \dfrac{(d-k)n}{k}\leq \sum_{i=0}^{d-k} B_i
 \leq \dfrac{dn}{2k},\quad B_0,\ldots, B_{d-k} \geq 0\, \Big\}.
\]
Then
\begin{align}
 \E(Y^2) &= \frac{1}{M(dn)}\,
 \sum_{\Bvec \in\mathcal{D}}
  \binom{n}{B_0,\ldots, B_{d-k},\, \frac{dn}{2k} - \sum B_i,\, \frac{dn}{2k} - \sum B_i,\, n-\frac{dn}{k} +
 \sum B_i} \nonumber \\
 & \hspace*{1cm} \times 
  \binom{d}{k}^{dn/k - 2\sum B_i}\, \Big(\sum (k-i)B_i\Big)!\, \Big(\frac{dn}{2}-\sum (k-i)B_i\Big)!\,
 \prod_{i=0}^{d-k} \x_i^{B_i} \nonumber \\
 &=  
    \sum_{\Bvec\in \mathcal{D}}\, J_n(\Bvec)
\label{EY2-raw}
\end{align}
where, for all $\Bvec\in\mathcal{D}$,
\[ J_n(\Bvec) = 
 \frac{n!\, (dn/2)!\, 2^{dn/2}\,
       \big(\sum (k-i)B_i\big)!\, \big(\dfrac{dn}{2}-\sum (k-i)B_i\big)!  }
  {(dn)!\, \big(\frac{dn}{2k} - \sum B_i\big)!^2\, \big(n - \frac{dn}{k} +\sum B_i\big)!}\, 
          \binom{d}{k}^{ dn/k- 2\sum B_i}\,  \,\,
  \prod_{i=0}^{d-k} \frac{\x_i^{B_i}}{B_i!}.
\]
Let $x\vee y$ denote $\max(x,y)$.  
Applying Stirling's formula in the form
\[ \log(N!) = N\log N - N + \dfrac{1}{2}\log(N\vee 1) + \dfrac{1}{2}\log 2\pi +
   O(1/(N+1)),\]
valid for all nonnegative integers, to  $J_n(\Bvec)$ shows that 
\begin{align}
 J_n(\Bvec) &= 
   \frac{\sqrt{n/2}\, \binom{d}{k}^{dn/k}}{(2\pi)^{(d-k+1)/2}\, d^{dn/2}\, n^{(d-2)n/2}}
 \cdot \sqrt{\frac{ \Big(\sum (k-i)B_i\Big) \Big(\frac{dn}{2} - \sum (k-i)B_i\Big)}
        { \Big(\frac{dn}{2k} - \sum B_i\Big)^2\, \Big(n-\frac{dn}{k} 
         + \sum B_i\Big)\, \prod_{i=0}^{d-k} B_i }}\nonumber \\
 & \times
  \frac{\Big(\sum (k-i)B_i\Big)^{\sum (k-i)B_i}\,
       \Big(\frac{dn}{2} - \sum (k-i)B_i\Big)^{dn/2 - \sum (k-i)B_i}}
       { \binom{d}{k}^{2\sum B_i}\, \Big( \frac{dn}{2k} - \sum B_i\Big)^{dn/k - 2\sum B_i}
         \, \Big(n - \frac{dn}{k} + \sum B_i\Big)^{n - dn/k + \sum B_i}}
  \,\, \cdot \,\,
    \prod_{i=0}^{d-k} \left(\frac{\x_i}{B_i}\right)^{B_i}
  \nonumber \\
 & \times
\left(1 + O\left(\frac{1}{\frac{dn}{2}  + 1 - \sum (k-i)B_i} \, +\,
   \frac{1}{\frac{dn}{2k} + 1 - \sum B_i} \right.\right.\nonumber \\
  & \left. \left. \hspace*{52mm} {} 
+ \, \frac{1}{n - \frac{dn}{k} + 1 + \sum B_i} \,
 +\, \sum_{i=0}^{d-k} \frac{1}{B_i + 1}\right)\right),\label{JnB-stirling}
\end{align}
where any factor in the denominator which equals zero should be replaced by~1.

\subsection{Laplace summation}

We now wish to apply Laplace's method to compute the asymptotic value of the
summation (\ref{EY2-raw}).  The following lemma, from Greenhill, Janson and 
Ruci{\'n}ski~\cite{GJR}, will help us perform our calculations.
(Some notation has been changed slightly to avoid clashes with notation used here.) 

\begin{lemma}[\protect{\cite[Lemma 6.3]{GJR}}] \label{lemma:laplace}
Suppose the following:
\begin{enumerate}
\item[\emph{(i)}] $\mathcal{L} \subset \mathbb{R}^m$ is a lattice with full rank $m$.
\item[\emph{(ii)}] $K \subset \mathbb{R}^m$ is a compact convex set with non-empty interior.
\item[\emph{(iii)}] $\varphi: K \to \mathbb{R}$ is a continuous function with a unique maximum at some interior point $z_0 \in K^\circ$.
\item[\emph{(iv)}] $\varphi$ is a twice continuously differentiable in a neighbourhood of $z_0$ and the Hessian $H_0:=D^2 \varphi(z_0)$ is strictly negative definite.
\item[\emph{(v)}] $\psi: K_1 \to \mathbb{R}$ is a continuous function on some neighbourhood $K_1 \subset K$ of $z_0$ with $\psi(z_0) > 0$.
\item[\emph{(vi)}] For each positive integer $n$ there is a vector $\ell_n \in \mathbb{R}^m$.
\item[\emph{(vii)}] For each positive integer $n$ there is a positive real number $A_n$ and a function $J_n: (\mathcal{L}+\ell_n) \cap nK \to \mathbb{R}$ such that, as $n \to \infty$,
\begin{align*}
& J_n(\ell) = O(A_n e^{n\varphi(\ell/n) + o(n)}), & \ell \in (\mathcal{L}+\ell_n) \cap nK, 
\intertext{and} 
& J_n(\ell) = A_n(\psi(\ell/n) + o(1))e^{n\varphi(\ell/n)}, & \ell \in (\mathcal{L}+\ell_n) \cap nK_1, 
\end{align*}
uniformly for $\ell$ in the indicated sets.
\end{enumerate}
Then, as $n \to \infty$,
\[
\sum_{(\mathcal{L} + \ell_n) \cap nK} J_n(\ell) \sim \frac{(2\pi)^{m/2}\, \psi(z_0)}{\det(\mathcal{L})\, \det(-H_0)^{1/2}}\, A_n \,n^{m/2} \,e^{n \varphi(z_0)}.
\]
\end{lemma}

\bigskip
	
To apply this lemma to (\ref{EY2-raw}), we first introduce the shorthand notation
\[ \beta_{\bvec} = \sum_{i=0}^{d-k} b_i,\qquad
   \gamma_{\bvec} = \sum_{i=0}^{d-k} (k-i) b_i.
\]
Then define the scaled domain
\[
 K = \Big\{ \bvec=(b_0,\ldots, b_{d-k})\in\mathbb{R}^{d-k+1} \,\, \Big|\,\, \dfrac{d-k}{k} \leq b_0 + \cdots + b_{d-k} \leq \dfrac{d}{2k},
    \quad b_0,\ldots, b_{d-k}\geq 0 \Big\}
\]
and let
\begin{align*}
A_n &= \frac{d^{-dn/2}}{\sqrt{2(2\pi n)^{d-k+1}}} \,  \binom{d}{k}^{dn/k},\quad\quad
\psi(\bvec) = \sqrt{\frac{\gamma_{\bvec}\, (d/2 - \gamma_{\bvec})}
		    {(d/(2k) - \beta_{\bvec})^2\, (1- d/k + \beta_{\bvec})\,
           \prod_{i=0}^{d-k} b_i}}, \nonumber \\
\varphi(\bvec) &= g(\gamma_{\bvec}) + g(d/2 - \gamma_{\bvec}) -
  2 g(d/(2k) - \beta_{\bvec}) - g(1-d/k + \beta_{\bvec}) - 2\beta_{\bvec}\log \binom{d}{k}\nonumber \\
 & \hspace*{6cm}
   + \sum_{i=0}^{d-k} b_i\log(\x_i) - \sum_{i=0}^{d-k} g(b_i),
	\end{align*}
	where $g(x) = x \log x$ for $x > 0$ and $g(0) = 0$. 

The following crucial result, which establishes property (iii) of Lemma~\ref{lemma:laplace}, will be proved in Section~\ref{s:maximisation}.
	
\begin{lemma} \label{lem:global-max}
Assume that $d,k$ are integers such that $d\geq 5$ 
and at least one of the following two conditions hold:
\begin{itemize}
	\item[\emph{(I)}] $d/2 +1<k \leq \kSSCM{d}$; 
	\item[\emph{(II)}] $d/2 < k \leq d/2+ \max\left\{1,\dfrac{\log d}{6}\right\}$.
\end{itemize}
Then the unique global maximum of $\varphi$ over $K$ occurs at the point $\bvec^*$ defined by
\begin{equation}
\label{bdef}
 b^*_i:= \left(\frac{d}{2k}\right)^2\, \binom{d}{k}^{-2} \x_i \quad \text{ for\, $i=0,\ldots, d-k$},
\end{equation}
where $\x_i$ is defined in \emph{(\ref{eq:tau-def})} for $i=0,\ldots, d-k$.
\end{lemma}

There is a natural interpretation for $b_i^*$ as the expected number of cells which are centres in both $O_1$ and $O_2$ and with $k-i$ points in $S_1\cap S_2$, if $O_1$ and $O_2$ are chosen independently and uniformly at random (from all possible orientations of all possible pairings in $\mathcal{P}_{n,d}$).  To see this, note that a cell is a centre in $O_1$ with probability $d/(2k)$, and hence a centre in both $O_1$ and $O_2$ 
with probability $\left(\frac{d}{2k}\right)^2$. Then for each cell which
is a centre in both $O_1$ and $O_2$, the probability that exactly $k-i$ points from this cell are in-points in both $O_1$ and $O_2$ is
\[ \binom{k}{k-i}\, \binom{d-k}{i}\, \binom{d}{k}^{-1} = \binom{d}{k}^{-2}\, \tau_i.\]

\bigskip

Next we prove some useful facts about $b^*$.

\begin{lemma}
\label{lem:identities}
The point $\bvec^*$ satisfies 
\[
\beta_{\bvec^*} = \sum_{i=0}^{d-k} b_i^* = \left(\frac{d}{2k}\right)^2,\qquad
\gamma_{\bvec^*} = \sum_{i=0}^{d-k} (k-i)\, b_i^* = \frac{d}{4},
\]
\begin{align*}
\varphi(\bvec^*) &=  \frac{d(k-2)}{2k}\log d + 2\log k - (d-2)\log 2 - \frac{2k-d}{k}\log(2k-d),\\
\psi(\bvec^*) &= \frac{2k^3}{(2k-d)^2}\, \bigg(\prod_{i=0}^{d-k} b_i^*\bigg)^{-1/2}.
		\end{align*}
Furthermore,
\[
\sum_{i=0}^{d-k} (k-i)^2\, b_i^* = \frac{d}{4} + \frac{d(k-1)^2}{4(d-1)}.
\]
\end{lemma}

\begin{proof}
Recall that $\x_i = \binom{d}{k-i,i,i,d-k-i}$ is the number of ways to choose
an ordered pair of subsets $A_1,A_2$, both of size $k$, from a set of size $d$,
so that $A_1\cap A_2 = k-i$.  Hence
$\sum_{i=0}^{d-k} \x_i = \binom{d}{k}^2$. The stated value of $\beta_{\bvec^*}$ is obtained
after multiplication by $\left(\frac{d}{2k}\right)^2\,\binom{d}{k}^{-2}$, 
by definition of $\bvec^*$.
Next, 
\begin{align*}
 \sum_{i=0}^{d-k} (k-i) \x_i &= \sum_{i=0}^{d-k} \frac{d!}{(k-i-1)!\, i!^2\, (d-k+i)!}\\
    &= d\, \sum_{i=0}^{d-k} \frac{(d-1)!}{(k-1-i)!\, i!^2\, (d-k+i)!}
     = d\binom{d-1}{k-1}^2,
\end{align*}
	arguing as above.  The stated value of $\gamma_{\bvec^*}$ follows, and now the expressions for $\varphi(\bvec^*)$ and
	$\psi(\bvec^*)$ can be obtained by direct substitution. 
	Finally
\begin{align*} 
\sum_{i=0}^{d-k} (k-i)^2 \x_i &= 
\left( \sum_{i=0}^{d-k} (k-i) \x_i\right) + 
  \left( \sum_{i=0}^{d-k} (k-i)_2 \x_i\right)\\
 &= d\, \binom{d-1}{k-1}^2 + \sum_{i=0}^{d-k} \frac{d!}{(k-i-2)!\, i!^2\, (d-k+i)!}\\
&= d\, \binom{d-1}{k-1}^2 + d(d-1)\binom{d-2}{k-2}^2,
\end{align*}
and multiplying by $\left(\frac{d}{2k}\right)^2\,\binom{d}{k}^{-2}$
completes the proof of the final identity.
\end{proof}

\bigskip

Let $H_*:=D^2 \varphi(\bvec^*)$ be the Hessian of $\varphi$ at $\bvec^*$.
To calculate $\det(-H_*)$ we prove the following lemma.  The main tools in the proof are the Matrix Determinant Lemma~\cite[equation (6.2.3)]{meyer}, which says
\begin{equation}
\label{eq:MDL}
 \det(A + \uvec\vvec^T) = (1 + \vvec^T\, A^{-1}\, \uvec)\, \det A,
\end{equation}
and the Sherman--Morrison Theorem~\cite[equation (3.8.2)]{meyer}),
which says
\begin{equation}
\label{eq:SM}
 (A + \uvec\vvec)^{-1} = A^{-1} - \frac{A^{-1}\, \uvec\vvec^T A^{-1}}{1 + \vvec^T A^{-1} \uvec},
\end{equation}
for any square matrix $A$ and vectors $\uvec$, $\vvec$ of the same dimension.

\begin{lemma}
\label{lem:gen-determinant}
Let $D$ be a diagonal $m\times m$ real matrix with entries $(d_1,\ldots, d_m)$,
and let $\wvec = (w_1,\ldots, w_m)^T$, $\vvec = (v_1,\ldots, v_m)^T\in\mathbb{R}^m$,
for some positive integer $m$.  Then
\[ \det(D + \vvec\vvec^T - \wvec\wvec^T) 
      = \left( \Big(1 + \sum_{i=1}^m v_i^2/d_i\Big)\Big(1 - \sum_{i=1}^m w_i^2/d_i\Big)
          + \Big( \sum_{i=1}^m v_iw_i/d_i\Big)^2\right)\, \prod_{j=1}^n d_j.
\]
\end{lemma}

\begin{proof}
Applying (\ref{eq:MDL}) twice, we obtain
\begin{align*}
\det(D  + \vvec\vvec^T - \wvec\wvec^T) 
   &= \det(D - \wvec\wvec^T)\, \left( 1 + \vvec^T(D - \wvec\wvec^T)^{-1}\, \vvec\right)\\
   &= \det(D)\, \bigg( 1 - \sum_{i=1}^m w_i^2/d_i\bigg)\, 
     \Big( 1 + \vvec^T(D - \wvec\wvec^T)^{-1}\, \vvec\Big).
\end{align*}
Then (\ref{eq:SM}) implies that
\begin{align*}
1 + \vvec^T(D - \wvec\wvec^T)^{-1}\, \vvec &=
1 + \vvec^T\left(D^{-1} + \frac{D^{-1}\wvec\wvec^T D^{-1}}{1 - \sum_{i=1}^m w_i^2/d_i}
                  \right)\vvec\\
&= 1 + \sum_{i=1}^m v_i^2/d_i + \frac{\big(\sum_{i=1}^m v_iw_i/d_i\big)^2}
                                                 {1 - \sum_{i=1}^m w_i^2/d_i},
\end{align*}
and combining these expressions completes the proof.
\end{proof}

\bigskip

Since we will need them to calculate the Hessian, we note here that
the partial derivatives of $\varphi$ are given by
\begin{align} \frac{\partial\varphi}{\partial b_i} &= -2\log \binom{d}{k} + \log \x_i + (k-i)\log(\gamma_{\bvec})
 - (k-i)\log(d/2 -\gamma_{\bvec}) - \log b_i \nonumber \\
  & \hspace*{2cm} {} + 2\log(d/(2k) - \beta_{\bvec})
  - \log(1 - d/k + \beta_{\bvec}) \label{partials}
\end{align}
for $i=0,\ldots, d-k$.  
Direct substitution shows that the point $\bvec^*$, given in (\ref{bdef}),
is a stationary point.

\begin{lemma}
\label{lem:detH}
	Suppose that $4k-d-2 > (2k-d)^2$, 
and recall that $H_*$ denotes the Hessian of $\varphi$ at the point $\bvec^*$. 
Then 
\[ \operatorname{det}(-H_*) = \frac{2k^2}{(d-1)(2k-d)^2}\,
  \big( 4k-d - 2 - (2k-d)^2\big)\, \prod_{i=0}^{d-k} \frac{1}{b_i^*} > 0.
\]
\end{lemma}

\begin{proof}
Recall that $\beta_{\bvec}$, $\gamma_{\bvec}$ are functions of
$\bvec$.
Differentiating (\ref{partials}) again gives
\[ \frac{\partial^2\varphi}{\partial b_i\, \partial b_j} 
     = \frac{(k-i)(k-j)}{\gamma_{\bvec}}  
          +  \frac{(k-i)(k-j)}{d/2 - \gamma_{\bvec}}
   - \frac{\mathbbm{1}(i=j)}{b_i} - \frac{2}{d/(2k) - \beta_{\bvec}}
    - \frac{1}{1-d/k + \beta_{\bvec}}
\] 
where the indicator function $\mathbbm{1}(i=j)$ equals 1 if $i=j$,
and equals zero otherwise.
Substituting $\bvec=\bvec^*$ and using Lemma~\ref{lem:identities},
it follows that
\[ -H_* = D + \vvec\vvec^T - \wvec\wvec^T\]
where $D$ is a $(d-k+1)\times (d-k+1)$ diagonal matrix with diagonal entries 
$d_i = 1/b_i^*$ for $i=0,\ldots, d-k$,
and the entries of $\vvec$ and $\wvec$ are given by
\[ v_i = \frac{2k}{2k-d}\, \sqrt{\frac{4k-d}{d}},\qquad
   w_i =  (k-i)\, \sqrt{8/d} \qquad \text{ for $i=0,\ldots, d-k$}.\]
Using Lemma~\ref{lem:identities} we calculate
\[
\sum_{i=0}^{d-k} v_i^2/d_i = \frac{d(4k-d)}{(2k-d)^2},\qquad
 \sum_{i=0}^{d-k} v_iw_i/d_i = \frac{k\sqrt{2(4k-d)}}{2k-d},\qquad
   \sum_{i=0}^{d-k} w_i^2/d_i = 2 + \frac{2(k-1)^2}{d-1}. 
\]
The claimed value for $\det(-H_*)$ is obtained by applying Lemma~\ref{lem:gen-determinant} 
and simplifying, and the lemma assumption implies that $\det(-H_*) > 0$.
\end{proof}

\subsection{Proof of Theorem~\ref{thm:main2} and Theorem~\ref{thm:main} }\label{ss:consolidation}

  The result for $(d,k)=(4,3)$  is established in~\cite{DP}. In the following  we assume that $d\geq 5$ and $k>d/2$, and show that the conditions of Theorem~\ref{thm:subgraph} hold under the assumptions of Theorem~\ref{thm:main2} and Theorem~\ref{thm:main}, where the random variables $Y$, $X_j$ are defined in 
  Section~\ref{s:structure}.

Condition (A1) follows from Bollob{\' a}s~\cite{bollobas1980probabilistic} with $\lambda_j= (d-1)^j/(2j)$ for $j\geq 1$,
while condition (A2) is established in Lemma~\ref{lem:all-cycles} with $\delta_j$ defined in (\ref{eq:delta}).
Next, (\ref{eq:P2}) holds by definition of $\kSSCM{d}$ if $k \leq \kSSCM{d}$,
	and by direct computation if $k \leq d/2 + \max\left\{1, \dfrac{\log d}{6}\right\}$.
	Hence the assumptions of Lemma~\ref{lem:magic-number} are
	satisfied, which implies that condition (A3) holds. 

To show that condition (A4) holds 
we will apply Lemma~\ref{lemma:laplace}. First we check the assumptions of
this lemma:
\begin{enumerate}
\item[(i)] Let $\mathcal{L} = \mathbb{Z}^{d-k+1}$, which is a lattice with rank $m = d-k+1$ and $\det(\mathcal{L}) = 1$.
\item[(ii)] The domain $K$ is compact and convex with a non-empty interior.
\item[(iii)] The function $\varphi: K \to \mathbb{R}$ is continuous, and has a unique global maximum $\bvec^*$,  by Lemma~\ref{lem:global-max}. 
Note that each value of $k$ in the range $d/2 < k \leq \kSSCM{d}$ is
		covered by assumption (I) or assumption (II).
	\item[(iv)] The function $\varphi:K \to \mathbb{R}$ is twice differentiable in the interior of $K$. Furthermore, the point $\bvec^*$ is the unique global
	maximum of $\varphi$ on $K$, by Lemma~\ref{lem:global-max}.
	Lemma~\ref{lem:detH} implies that $H_*$ is nonsingular, 
	using the fact (proved above) that (\ref{eq:P2}) holds. 
	Hence $H_*$ is strictly negative definite.
\item[(v)] Let $K_1$ be an open ball around $\bvec^*$ with radius small enough
to guarantee that $K_1 \subset K$. 
Then the function $\psi: K_1 \to \mathbb{R}$ is a continuous function 
and $\psi(\bvec^*) > 0$, by Lemma~\ref{lem:identities}.
\item[(vi)] Let $\ell_n$ be the zero vector in $\mathbb{Z}^{d-k+1}$, for each $n$. 
\item[(vii)] This condition follows from (\ref{JnB-stirling}).
\end{enumerate}
Thus, we can apply Lemma \ref{lemma:laplace} to see that
\begin{align*}
\E Y^2 &\sim \frac{(2\pi)^{(d-k+1)/2} \,\, \psi(\bvec^*)}{\det(\mathcal{L})\, \det(-H_*)^{1/2}} \,\, A_n \, n^{(d-k+1)/2} \, e^{n\, \varphi(\bvec^*)} \\
&= 
 \frac{k^2\, \sqrt{d-1}}{(2k-d)\, \sqrt{4k-d-2-(2k-d)^2}}\,\,
    \left(\frac{\binom{d}{k}^d\, k^{2k}}{2^{k(d-2)}\, d^d\, (2k-d)^{2k-d}}\right)^{n/k}.
\end{align*}
Dividing by $(\E Y)^2$, using (\ref{expectation}), proves that
\[
\frac{\E Y^2}{(\E Y)^2} \sim \sqrt{\frac{d-1}{4k-d-2-(2k-d)^2}}.
\]
This shows that condition (A4) of Theorem~\ref{thm:subgraph} holds,
by Lemma~\ref{lem:magic-number}.
	
Since $\delta_j > -1$ for all $j\geq 1$, we conclude from Theorem~\ref{thm:subgraph} that a.a.s.\ $Y > 0$. By (\ref{negative-useful}), this completes the proof of Theorem~\ref{thm:main} and Theorem~\ref{thm:main2}.

\section{Unique global maximum }\label{s:maximisation}

In this section we prove Lemma~\ref{lem:global-max},   
showing that $\bvec^*$ is the unique global maximum of $\varphi$ in $K$.

\subsection{No global maxima on the boundary}\label{s:proof-Lemma4.2}

First we consider the boundary of the domain $K$.

\begin{lemma}
	Let the assumptions of Lemma~\emph{\ref{lem:global-max}} hold.
	Suppose that we know that $\xvec^*$ is the unique maximum of $\varphi$
	in the interior of $K$.  Then $\varphi(\bvec^*) > \varphi(\bvec)$ for all
	points $\bvec$ on the boundary of $K$. 
	\label{lem:boundary}
\end{lemma}

\begin{proof}
	First we consider the point $\avec = (\nfrac{d}{2k},0,\ldots, 0)$;
	that is, $a_{0}=\nfrac{d}{2k}$ and all other $a_i$ are zero.  
	Then 
	$\beta_{\avec} = \dfrac{d}{2k}$ and $\gamma_{\avec} = \dfrac{d}{2}$.
	Plugging these values into $\varphi$, we find that 
	\[
	\varphi(\avec) = \frac{d(k-1)}{k}\log(d) + \log(k) - \frac{2k-d}{2k}\log(2k-d)
	- \frac{(d-2)}{2}\log 2 - \frac{d}{2k}\log\binom{d}{k},
	\]
	using the fact that $\x_{0}= \binom{d}{k}$.
	Then $\varphi(\bvec^*) > \varphi(\avec)$ if and only if 
	\[ d\, \log \binom{d}{k} + 2k\log k > d\log d + k(d-2)\log 2 + (2k-d)\log(2k-d).\]
	But this inequality is equivalent to 
	(\ref{orientation-expectation-threshold}). 
	Note that $\kSSCM{d}\leq \kind{d}$ by definition, and $\lceil \frac{d+1}{2}\rceil \leq \kind{d}$ since (\ref{eq:k+}) holds for this value of $k$, by
	direct computation. Furthermore, if $d\geq 404 > e^6$, then
	\[ d/2 + \dfrac{\log d}{6} \leq \kind{d}.\]
(To see this, note that (\ref{eq:k+}) holds when $k=1.01 d/2$ and
$\frac{\log d}{6} < d/200$ when $d\geq 404$.) 
	Hence the conditions of Lemma~\ref{L:k-star} hold if either
	(I) or (II) hold, so (\ref{orientation-expectation-threshold}) holds
	and hence $\avec$ is not a global maximum of $\varphi$.
	
	Next, suppose that $\bvec\neq\avec$ satisfies
	$\sum_{i=0}^{d-k} b_i= \frac{d}{2k}$.  
	Then each of $\gamma_{\bvec}$, $\dfrac{d}{2}-\gamma_{\bvec}$
	and $1-\dfrac{d}{k} + \beta_{\bvec}$ are positive, noting that
	$\gamma_{\yvec}=\dfrac{d}{2}$ if and only if $\yvec=\avec$.
	Choose $i$ such that $b_i > 0$.
	In the expression for $\partial \varphi/(\partial b_i)$,
	every term is bounded except for the term corresponding to
	$\dfrac{d}{2k}-\beta_{\bvec}$, and this term contributes $-\infty$.
	Hence reducing $b_i$ slightly while holding all other entries steady
	will lead to an increase in $\varphi$, showing that no point on the
	boundary $\sum_{i=0}^{d-k} b_i = \dfrac{d}{2k}$ can be a local maximum of $\varphi$.
	
	Now, suppose that $\dfrac{d-k}{k} < \sum_{i=0}^{d-k} b_i < \dfrac{d}{2k}$ and 
	$b_{\ell}=0$ for some 
	$\ell\in \{0,\ldots, d-k\}$.  Take some $j\in \{0,\ldots, d-k\}$ such that
	$b_j > 0$, and replace $(b_\ell,b_j)$ by $(b_\ell + \varepsilon,b_j-\varepsilon)
	= (\varepsilon, b_j-\varepsilon)$.  This leaves $\beta_{\bvec}$ unchanged.
	Let $\widehat{\varphi}(\varepsilon) = \varphi(\bvec + \varepsilon\, \evec_\ell 
	- \varepsilon\, \evec_j)$, where $\evec_i$ denotes the $i$th standard
	basis vector for $i=\ell,j$.
	By the chain rule, 
	\[ \widehat{\varphi}'(0) = 
	\log \x_\ell - \log \x_j + (\ell-j) \log(\gamma_{\bvec}) - (\ell-j)\log(d/2-\gamma_{\bvec}) - \log(b_\ell) + \log(b_j).
	\]
	The only unbounded term is $-\log b_\ell$, which equals $+\infty$.  Therefore,
	replacing $\bvec$ by $\bvec + \varepsilon \evec_\ell - \varepsilon \evec_j$
	gives a new vector in $K$ with a higher value of $\varphi$, for some
	sufficiently small $\varepsilon > 0$. This proves that no vector in $K$
	with a zero entry can be a local maximum of $\varphi$.
	
	Finally, suppose that $\bvec$ satisfies $\sum_{i=0}^{d-k} b_i = \dfrac{d-k}{k}$
	and that all entries of $\bvec$ are positive. Then   the only unbounded term
	in $\partial \varphi/(\partial b_1)$ is $-\log(1-\dfrac{d}{k} + \beta_{\bvec})$, 
	which equals
	$+\infty$.  Hence increasing $b_1$ slightly leads to a new vector which
	lies inside $K$ and has a higher value of $\varphi$.  Therefore the boundary
	given by $\sum_{i=0}^{d-k} b_i = \dfrac{d-k}{k}$ does not contain any local 
	maximum of $\varphi$.  This completes the proof.
\end{proof}

\subsection{Stationary points}\label{ss:interior}

In this section we write $\mu=\binom{d}{k}$.
Setting all partial derivatives equal to zero using (\ref{partials}), we find that at any
stationary point of $\varphi$ in the interior of $K$,
\begin{equation}
\label{reformulate}
   b_i = \frac{\x_i\, \big(\frac{d}{2k}-\beta\big)^2 }
              {\mu^2\, \big(1-\frac{d}{k} + \beta\big)}\, 
	          \left(\frac{\gamma}{d/2-\gamma}\right)^{k-i}
\end{equation}
for $i=0,\ldots, d-k$.   Here $(\beta,\gamma) = (\beta_{\bvec},\gamma_{\bvec})$.
Alternatively, we may treat $\beta,\gamma$ as the variables and use
(\ref{reformulate}) to define the values of $b_i$, whenever $\bvec$ is
a stationary point of $\varphi$ in the interior of $K$.

Next we write $\beta$, $\gamma$ in terms of a single variable $x$.
Define $x=\frac{\gamma}{d/2-\gamma}$. Rearranging for $\gamma$ gives
\begin{equation}
\label{new-gamma}
   \gamma = \gamma(x) = \frac{xd}{2(x+1)}.
\end{equation}  
Observe that $\gamma^* = \gamma(1)$.
The polynomial $f$ defined in (\ref{expan0}) can also be written as
\[ f(x) = \sum_{i=0}^{d-k} \frac{\x_i}{\mu^2}\, x^{k-i}. \]
Summing (\ref{reformulate}), it follows that at any
stationary point, 
\begin{equation}
\label{beta-identity}
    \beta =\beta(x) = \frac{\big(\frac{d}{2k}-\beta\big)^2}
              {\big(1-\frac{d}{k} + \beta\big)} \, f(x).
\end{equation}
Note that $\beta^* = \beta(1)$.
Next, using $\gamma = \sum_{i=0}^{d-k} (k-i)b_i$ and substituting
(\ref{reformulate}) and (\ref{new-gamma}) into (\ref{beta-identity}) gives
\begin{align*}
 \frac{xd}{2(x+1)}  &= 
   \frac{\big(\frac{d}{2k}-\beta\big)^2}
              {\big(1-\frac{d}{k} + \beta\big)} \, 
    \sum_{i=0}^{d-k} \frac{(k-i)\, \x_i}{\mu^2}\, x^{k-i}\\
  &= 
   \frac{\big(\frac{d}{2k}-\beta\big)^2}
              {\big(1-\frac{d}{k} + \beta\big)} \, 
  x f'(x)\\
  &= \frac{\beta x f'(x)}{f(x)},
\end{align*}
using (\ref{beta-identity}) for the final equality.
Dividing through by $x$, we obtain the identity
\begin{equation}
\label{f-identity}
   \frac{d}{2(x+1)} = \frac{\beta f'(x)}{f(x)}.
\end{equation}
We want to show that this identity has a unique solution on $(0,\infty)$ at $x=1$.

Let
\[ y=y(x)=\frac{(x+1)f'(x)}{k}.\]
Then (\ref{f-identity}) says that
$\beta = \frac{d}{2k}\cdot \frac{f}{y}$, and we need $\frac{d-k}{k} < \beta < \frac{d}{2k}$ for the
corresponding vector $\bvec$ to belong to $K^\circ$.  Therefore
we are interested in values of $x$ for which 
\[ \frac{2(d-k)}{d}\, y < f < y.\]
  We can rewrite (\ref{beta-identity}) as
   $(y-f)^2 = f - \frac{2(d-k)}{d}\, y$, or in terms of $x$:
   \begin{equation}\label{eq1}
	   \left(\frac{(x+1)f'}{k}-f\right)^2 = f -  \frac{2(d-k)\, (x+1) \, f'}{d k}.
   \end{equation}
 We need to show that $x=1$ is the unique solution of (\ref{eq1})
     on the interval $(0,1]$ under the constraints $\frac{2(d-k)}{d}\, y < f < y$.
   
Next, we solve~\eqref{eq1} as a quadratic equation in $y$.
We can rule out one of the branches using the constraint $y > f$,
leading to  
\[
 y =  f - \frac{d-k}{d} + \frac{\sqrt{(d-k)^2 + d(2k-d)\, f}}{d} = (1 +\eta) f
\]
where $\eta=\eta(x)$ is defined in (\ref{eq:eta-def}).
It suffices to prove that the equation 
\begin{equation}  
\frac{(x+1)\, f'(x)}{k} = (1 + \eta(x))\, f(x)
\label{diff-eq2}
\end{equation}
has a unique solution on $(0,\infty)$ at $x=1$.   In the following subsections, we give  two separate arguments under assumptions (I) and (II).  
First we collect some useful identities and facts about $f$.

Note that $f$ is an increasing function of $x$ when $x>0$, and hence
$0< f(x) < 1$ when $x\in (0,1)$.
We will also use the expansion of $f$ around $1$, which can be
obtained using inclusion-exclusion (see Wilf~\cite[Section 4.2]{Wilf}):
\begin{equation} \label{expan1}
	f(x) = \mu^{-1}\sum_{j=0}^{k} \binom{k}{j} \binom{d-j}{k-j}  (x-1)^j.
\end{equation}
For later use, observe that (\ref{expan1}) gives 
\begin{equation}
	\label{eq:f-derivs}
	f(1) = 1,\qquad f'(1) = \frac{k^2}{d},\qquad f''(1) = \frac{k^2(k-1)^2}{d(d-1)},
\end{equation}
and 
\begin{align}
	y(x) &=  \frac{(x-1) f'(x)}{k} + \frac{2 f'(x)}{k} \nonumber \\
	&=   \mu^{-1}\sum_{j=0}^k 
	\binom{k}{j} \binom{d-j}{k-j} \left(\frac{j}{k} + \frac{2(k-j)^2}{k(d-j)} \right) (x-1)^j.\label{eq:y}
\end{align}

\subsection{Proof of Lemma \ref{lem:global-max}. Part (I)}

Here, we prove  Lemma \ref{lem:global-max}  under assumption (I), which is  
 $d/2 +1<k \leq \kSSCM{d}$.  From Table~\ref{table:values}, we find that 
 $d \geq 9$. Then inequality~\eqref{eq:P2} implies that 
 $k \leq (d + 1 + \sqrt{d-1})/2$, and hence $3k \leq 2d$.
Using condition \eqref{eq:P1} and Lemma~\ref{lem:boundary}, to complete
the proof of  Lemma \ref{lem:global-max} we must show that the equation
\eqref{diff-eq2} has no solution on 
$ 
\left( 0, \left(1+\frac{(2k-d)^2 d}{k (d-k)(4k-d-2-(2k-d)^2)}\right)^{-1}\right)
\cup \left(\dfrac{5k-2d}{d-k},\infty\right)$.   
First we consider the large values of $x$.

  \begin{lemma}\label{L1}
  Suppose that $d\geq 9$ and $\dfrac d 2  +1 < k\leq \dfrac {2d}{3}$.
 If  $x\geq  \dfrac{5k-2d}{d-k}$, then $(x+1)f'< kf$. 
  \end{lemma}

  \begin{proof}
	  Using (\ref{expan1}) and \eqref{eq:y}, we can write
        \begin{align}
        	f - \frac{(x+1)f'}{k} 
        	 &=  \mu^{-1}   \sum_{j=0}^k 
        	\binom{k}{j} \binom{d-j}{k-j} \left(1-\frac{j}{k} - \frac{2(k-j)^2}{k(d-j)} \right) (x-1)^j\label{eq:f-middle}\\
		 &=  \frac{1}{k \mu}   \sum_{j=0}^k 
		\binom{k}{j} \binom{d-j-1}{k-j-1} \, (d-2k+j)\, (x-1)^j\nonumber\\
        	 &\geq      \sum_{\ell =1}^{2k-d}   \ell (x-1)^{2k-d} 
        	 \left( c_{2k-d+\ell} (x-1)^{\ell} -  c_{2k-d-\ell} (x-1)^{-\ell}\right)
		 \label{eq:paired-sum}  
        \end{align}
        where $c_j = \frac{1}{\mu k} \binom{k}{j} \binom{d-j-1}{k-j-1}$ for $j=0,\ldots, k$.  
	For the inequality, note that our assumptions imply that
	  $4k-2d\leq k$, so the right hand side of (\ref{eq:paired-sum}) is
	  a sum of the $j=0,\ldots, 4k-2d$ terms of (\ref{eq:f-middle}).
	  The omitted terms, with $4k-2d < j\leq k$, all
	  make a positive contribution to (\ref{eq:f-middle}).

We claim that
\begin{equation}
	\frac{c_{2k-d-\ell}\cdot c_{2k-d+\ell-1}}{c_{2k-d-\ell+1}\cdot c_{2k-d+\ell} } \leq  \frac{8}{9}(x-1)^2 \quad \text{ for $\ell = 1,\ldots, 2k-d-1$},
 \label{eq:4cs}
\end{equation}
	and note that this implies that
for all $\ell=1,\ldots, 2k-d-1$,
\begin{equation}\label{eq98} \frac{c_{2k-d+\ell}}{c_{2k-d-\ell}}\, (x-1)^{2\ell} \geq 
\frac98 \cdot \frac{c_{2k-d+\ell-1}}{c_{2k-d-\ell+1}}\, (x-1)^{2(\ell-1)} \geq \cdots \geq  \left(\frac98\right)^{\ell}.
\end{equation}
Hence, if the claim holds, then
all terms in (\ref{eq:paired-sum}) with $\ell \leq 2k-d-1$ are positive.
To bound the remaining term with $\ell = 2k-d$,   observe that 
\begin{align*}
		\frac{(2k-d) c_0}{(2k-d -1) c_1 (x-1)}  &=\frac{(2k-d)(d-1)}{(2k-d-1) k(k-1) (x-1)}  \leq \frac{(d-1)(d-k)}{3 k(k-1)(2k-d-1)},\\
		\frac{(2k-d) c_{4k-2d} (x-1) }{(2k-d -1) c_{4k-2d-1}}
		&= \frac{ (2d-3k+1)(2d-3k)(x-1)}{2(2k-d-1)  (3d-4k)}
		\geq \frac{3 (2d-3k) (2d-3k+1)}{2(d-k)(3d-4k)}.
\end{align*}
Under our assumptions, we have that 
\begin{equation}\label{eq:cases} 
	\left(1+\frac{3 (2d-3k) (2d-3k+1)}{2(d-k)(3d-4k)}\right)  \left(\frac{9}{8}\right)^2
	>1+ \frac{(d-1)(d-k)}{3 k(k-1)(2k-d-1)}.
\end{equation}
To see this, suppose first that $2k -d \geq 4$. Then bounding $d -1 \leq 2(k-1)$ and $d-k \leq k$,
the right-hand side of \eqref{eq:cases} is at most $1+ \frac{2}{9} < \left(\frac{9}{8}\right)^2$.  
Similarly, if  $2k-d = 3$  and $k \geq 7$, then we bound the right-hand side  of \eqref{eq:cases}
by $\frac43$ and estimate
\[
1 + \frac{3 (2d-3k)(2d-3k+1)}{2(d-k)(3d-4k)}    = 
1+ \frac{3 (k-6)(k-5)}{2(k-3)(2k-9)} \geq \frac{4}{3}
>	  \frac43 \left(\frac{8}{9}\right)^{2}.
\]
Here we use the fact that 
	  $3(k-6)(k-5)/\big(2(k-3)(2k-9)\big)$ is monotonically
	  increasing with $k$.
Finally, if  $2k-d = 3$  and $k = 6$,
then $d=9$ and direct substitution shows that  
	  the right-hand side of \eqref{eq:cases}  equals $1 + \frac{2}{15} <  \left(\frac{9}{8}\right)^2$.  Hence (\ref{eq:cases}) holds in all cases.

 Combining the terms in (\ref{eq:paired-sum}) for $\ell  \in \{2k-d, 2k-d-1\}$
 and using \eqref{eq98}, \eqref{eq:cases}, we see that
 \begin{align*}
  &\frac{1}{2k-d-1}  \sum_{\ell  \in \{2k-d, 2k-d-1\} }   \ell (x-1)^{2k-d} 
 	\left( c_{2k-d+\ell} (x-1)^{\ell} -  c_{2k-d-\ell} (x-1)^{-\ell}\right)
 	\\
 	&\geq    \left(1+\frac{3 (2d-3k)(2d-3k+1)}{2(d-k)(3d-4k)}\right)  c_{4k-2d-1} (x-1)^{4k-2d-1}
 	 \\  &\hspace{60mm}-   \left(1+ \frac{(d-1)(d-k)}{3 k(k-1)(2k-d-1)}\right)  c_{1} (x-1)\\
 	 &> \left(1+ \frac{(d-1)(d-k)}{3 k(k-1)(2k-d-1)}\right)   \left(
 	 \left(\frac89 \right)^2c_{4k-2d-1} (x-1)^{4k-2d-1} -  c_{1} (x-1)\right)
 	 \\ &\geq 0.
 	\end{align*}
 Thus, the sum in (\ref{eq:paired-sum})  is positive and the lemma is proved.

It remains to establish claim \eqref{eq:4cs}.
Using   $\frac{c_j}{c_{j+1}} = \frac{(j+1)(d-j-1)}{(k-j)(k-j-1)}$, we find that 
        \begin{align}
       \frac{c_{2k-d-\ell}}{c_{2k-d-\ell+1}}
       \cdot &
       \frac{c_{2k-d+\ell-1}}{c_{2k-d+\ell}} 
       \nonumber \\
      &=  \frac{(2(d-k)-\ell)(2k-d+\ell)}{(d-k)^2-\ell^2} \cdot 
        	 \frac{(2(d-k)+\ell-1)(2k-d-\ell+1)}{(d-k)^2 -(\ell-1)^2}\label{eq:4cs-simplified}.
\end{align}		
Since $3k \leq 2d$, it follows that   the denominators on the right hand
side of (\ref{eq:4cs-simplified}) are always positive for $\ell \leq 2k-d-1$.
Now
\begin{align*} 
      \frac{(2(d-k)+\ell-1)(2k-d-\ell+1)}{(d-k)^2-(\ell-1)^2} &\leq
\frac{(2(d-k)+\ell-1)(2k-d-\ell+1)}{(d-k)^2-\ell^2}\\
       &\leq \frac{2(2(d-k)+\ell)(2k-d-\ell)}{(d-k)^2-\ell^2}, 
       \end{align*}
and hence by (\ref{eq:4cs-simplified}),  we obtain
\[ 
       \frac{c_{2k-d-\ell}}{c_{2k-d-\ell+1}}
       \cdot 
       \frac{c_{2k-d+\ell-1}}{c_{2k-d+\ell}} \leq
\frac{2\big(4(d-k)^2 -\ell^2\big)\big((2k-d)^2-\ell^2\big)}{\big((d-k)^2 - \ell^2\big)^2}.\]
Next,  using $3k \leq 2d$, we observe that
\[
	\left( 1 - \frac{\ell^2}{4(d-k)^2}\right)\left( 1 - \frac{\ell^2}{(2k-d)^2}\right)
	\leq 1 - \frac{\ell^2}{(d-k)^2}.
\]
 Hence for $\ell=1,\ldots, 2k-d-1$, using the lower bound on $x$, we have
\[ 
\frac{2\big(4(d-k)^2 -\ell^2\big)\big((2k-d)^2-\ell^2\big)}{\big((d-k)^2 - \ell^2\big)^2}
\leq \frac{8(2k-d)^2}{(d-k)^2}   \leq \frac89 (x-1)^2.
\]
This completes the proof.
  \end{proof}

 To prove  Lemma \ref{lem:global-max}  under assumption (I), 
 it only remains to show that equation
 \eqref{diff-eq2} has no solution when 
 $0 <x < \left(1+\frac{d (2k-d)^2 }{k(d-k)(4k-d-2-(2k-d)^2)}\right)^{-1}$.   
 Since $\eta$ is monotonically decreasing and $\eta(0) = \frac{2k-d}{2(d-k)}$, it is sufficient to show that, for such $x$,
 \begin{equation}\label{eq:sufficient}
 		\frac{(x+1)f'}{k} \geq \frac{d}{2(d-k)} f.
 \end{equation}
 Let  $t = 1/x$ and  
 \begin{equation}\label{g-def}
 g(t) = f(t^{-1})t^k =\mu^{-1} \sum_{j=0}^{d-k} \binom{k}{j} \binom{d-k}{j} t^{j}.
 \end{equation}
 Note that $g(t)$ is the PGF of the hypergeometric distribution with parameters $(d,k,d-k)$. We will also use the expansion of 
 $g$ around 1, given by
 \begin{equation}\label{g-exp}
 	g(t) = \mu^{-1} \sum_{j=0}^{d-k}   \binom{k}{j} \binom{d-j}{k}(t-1)^j.
 \end{equation}
 Now
 \begin{equation}
 	\label{eq:ft-deriv}
 	f'(t^{-1}) = - \frac{g'(t)}{t^{k-2}} + \frac{k g(t)}{t^{k-1}}.
 \end{equation}
Rewriting equation \eqref{eq:sufficient} in terms of $g(t)$ and $t$, we get 
\[	
	\frac{1+t^{-1}}{k} \left(- \frac{g'(t)}{t^{k-2}} + \frac{k g(t)}{t^{k-1}}\right)
	\geq \frac{d}{2(d-k)} \cdot \frac{g(t)}{t^k},
\]
which is equivalent to the inequality 
\begin{equation}\label{def:A}
		A(t) = (t+1)g -  \frac{d}{2(d-k)} g -  \frac{t(t+1) g'}{k} \geq 0
	\end{equation}
for $t \geq 1+\frac{d (2k-d)^2 }{k(d-k) (4k-d-2-(2k-d)^2)}$.

 \begin{lemma}\label{l:Claim1}
 	If \emph{(\ref{eq:P2})} holds, then 
   $A''(t) > 0$ for all   $t \geq 1$ and 
   \[
   	A'(1) =     \frac{k (4k-d-2 - (2k-d)^2)}{2d(d-1)}.
   \]
   	\end{lemma}
\begin{proof}
		 It is sufficient to  show that 
		$A = A(1) + \sum_{j=1}^{d-k+1} a_j (t-1)^j$ where  $a_j >0$ for all $j$.
		Using  the expansion from \eqref{g-exp}, we have
		\begin{align*}
		(t+1)g &= 2g + (t-1)g 
		\\ &=  2  g(1)+ \mu^{-1}\sum_{j=1}^{d-k+1}  
		\left(2 \binom{k}{j} \binom{d-j}{k} +   \binom{k}{j-1} \binom{d-j+1}{k}\right) (t-1)^j
		\end{align*}
		and 
		\begin{align*}
			t(t+1)g'&= 2 g'+  3(t-1)g' + (t-1)^2 g' 
			\\
			&= 2g'(1) +  \mu^{-1} \sum_{j=1}^{d-k+1} 
			\bigg( 
			2(j+1) \binom{k}{j+1} \binom{d-j-1}{k} +
			3 j \binom{k}{ j} \binom{d-j}{k} \\
			&\hspace{50mm}+ (j-1) \binom{k}{j-1} \binom{d-j+1}{k}
			\bigg) (t-1)^j.
		\end{align*}
		Thus, we find  that 
		\[
		a_{d-k+1} = \mu^{-1} \binom{k}{d-k} \left( 1-  \dfrac{d-k}{k} \right) > 0
		\]
		and, for $j=1\ldots,d-k$, we have 
		$a_{j} = \mu^{-1} \binom{k}{j} 
		\binom{d-j}{k}  h_j$, where
		\begin{align*}
			h_j &= 2 + \dfrac{j(d-j+1)}{(k-j+1)(d-k-j+1)} 
			- \dfrac{d}{2(d-k)} - \dfrac{2(k-j)(d-k-j)}{k(d-j)}
			- \dfrac{3j}{k} - \dfrac{j(j-1) (d-j+1)}{k(k-j+1)(d-k-j+1)}
			\\
			&=  2 +  \dfrac{j(d-j+1)}{k (d-k-j+1)}  -  
			\dfrac{d}{2(d-k)} - \dfrac{3j}{k}  - \dfrac{2(k-j)(d-k-j)}{k(d-j)} 
			\\
			&= 2   + \dfrac{j}{d-k-j+1} -    \dfrac{d}{2(d-k)}   - \dfrac{2(d-k)}{d-j}.
		\end{align*}
		Direct calculations show that 
		\begin{equation}\label{e:c1}
			h_1 = \frac{4k -d -2 - (2k-d)^2}{2(d-1)(d-k)}.
		\end{equation}
		Since (\ref{eq:P2}) holds we see that $h_1>0$.
		For any positive $1\leq z\leq d-k$, let 
		\[
		h(z)  =  2   + \dfrac{z}{d-k-z+1} -    \dfrac{d}{2(d-k)}   - \dfrac{2(d-k)}{d-z}.
		\] 
		Note that  $d-z \geq 2(d-k-z+1)$ for all $z\geq 1$. Therefore,
		\[
		h'(z) = \dfrac{d-k+1}{(d-k-z+1)^2} -  \dfrac{2(d-k)}{(d-z)^2} 
		\geq \frac{4(d-k+1)-2(d-k)}{(d-z)^2}> 0,
		\]
		which implies that $h_j = h(j) \geq h(1) > 0$ for all $j=2,\ldots, d-k+1$.  This completes the proof of the first statement, and the expression
		for $A'(1)$ follows as $A'(1) = a_1$. 
	\end{proof}

 We find from \eqref{g-exp} that 
 \begin{equation}
	 \label{eq:g-at-1}
	 g(1)=1 \quad \text{ and } \quad g'(1) = k(d-k)/d.  
 \end{equation}
 Thus, $A(1) = -\frac{(2k-d)^2 }{2d(d-k)}$. 
 Using  Lemma \ref{l:Claim1}, we find that, for $t \geq 1+\frac{(2k-d)^2 d}{k (d-k)(4k-d-2-(2k-d)^2)}$,
\begin{equation*} 
	A(t) \geq A(1) + (t-1)A'(1)
	 \geq  -\dfrac{(2k-d)^2 }{2d(d-k)}+ \dfrac{(2k-d)^2 d}{k (4k-d-2-(2k-d)^2)}\, A'(1) \geq  0.
\end{equation*}
This establishes \eqref{def:A} and completes the proof of Lemma \ref{lem:global-max}  under assumption (I).

\subsection{Proof of Lemma \ref{lem:global-max}.  Part (II) for  $k = \lceil (d+1)/2 \rceil$. }\label{ss:proof-B-k0}

 In this section we prove  Lemma \ref{lem:global-max} for the case when $k=k_0(d)$ defined by
\[  k_0(d) =  \begin{cases} (d+1)/2 & \text{ when $d$ is odd,}\\
	d/2+1 & \text{ when $d$ is even.}
\end{cases}
\]
However, some of the lemmas in this section will be proved under weaker
assumptions on $k$, so that we can reuse them in Section~\ref{s:large-d}. 

\begin{lemma}
	Suppose that $d\geq 5$ and $k= k_0(d)$.  If $x>1$, then
	\[ 
	\frac{(x+1)\, f'(x)}{k} < \big(1 + \eta(x)\big)\, f(x).
	\]
	\label{L1-smalld}
\end{lemma}

\begin{proof}
	Since equality holds in (\ref{diff-eq2}) when $x=1$, it
	suffices to prove that the derivative of the LHS is strictly
	smaller than the derivative of the RHS whenever $x>1$.
	The left hand side of (\ref{diff-eq2}) equals $y(x)$.
	Define 
	\[ R_1(x) =  f'(x),\qquad R_2(x) = \frac{(2k-d)\, f'(x)}{2 \sqrt{(d-k)^2 + d(2k-d)\, f(x)}}.\]
	Then the derivative of the left hand side of (\ref{diff-eq2}) is given by
	\[   y'(x) = \frac{1}{k}\, \big( (x+1)\, f''(x) + f'(x)\big),\]
	while the derivative of the right hand side equals $R_1(x) + R_2(x)$.
	We claim that
	\begin{itemize}
		\item[(a)] $y'(1) < R_1(1) + R_2(1)$;
		\item[(b)] $(R_1 - y')^{(\ell)}(1) \geq 0$ for all $\ell\geq 1$. That is, the $\ell$-th order derivatives of $R_1(x) - y'(x)$ are all nonnegative at $x=1$; 
		\item[(c)] $R_2'(x) > 0$ for all $x > 1$. 
	\end{itemize}
	Before establishing these facts, we show why they complete the
	proof.  From (b), we conclude that $R_1(x) - y'(x)$ is an increasing
	function of $x$ on $(1,\infty)$, as all non-constant coefficients of 
	the Taylor expansion of $R_1(x) - y'(x)$ around
	$x=1$ are nonnegative. 
	From (c) we conclude that $R_2(x)$ is
	a strictly increasing function of $x$ on $(1,\infty)$.  Then (a) 
	implies that $y'(x) < R_1(x) + R_2(x)$ for
	all $x>1$, as required. 
	Thus it remains to establish (a), (b) and (c).
	
	Using (\ref{eq:f-derivs}), we calculate that
	\[
		R_1(1) = f'(1) = \frac{k^2}{d},\quad R_2(1) = \frac{(2k-d)k}{2d},
		\quad y'(1) = \frac{k\big(d-1+2(k-1)^2\big)}{d(d-1)}.
		\]
	Direct substitution shows that
	\[ (4k-d-2)(d-1)  > 4(k-1)^2\]
	when $k = k_0(d)$ and $d\geq 5$.  Hence (a) holds.
	
	Since $(f-y)' = R_1- y'$ and
	\[ 1 - \frac{j}{k} - \frac{2(k-j)^2}{k(d-j)}  = \frac{(k-j)(d-2k+j)}{k(d-j)},\]
	it follows from (\ref{expan1}) and (\ref{g-exp}) that all non-constant coefficients of $R_1(x)-y'(x)$ are nonnegative, noting that the non-constant coefficients of $R_1 - y'$ corresponds to terms with $j\geq 2$.
	By Taylor's Theorem, we conclude that~(b) holds. 
	
	Finally we consider~(c).  Note that
	\begin{align*}
		\left( R_2^2\right)' &= \frac{(2k-d)^2\, f'}
		  {d^2\, \left((d-k)^2 + d(2k-d)\, f\right)^2}\, \left( f''\, \big( 2(d-k)^2 + 2d(2k-d)\, f \big) - d(2k-d)\, (f')^2\right).
	\end{align*}
	Furthermore, $R_2(x)>0$ for all $x\geq 1$, by definition.
	Hence  $R_2'(x) > 0$ if and only if 
	\begin{equation}  f''(x)\left( 2(d-k)^2 + 2d(2k-d)\, f(x)\right) 
		- d(2k-d)\, \big(f'(x)\big)^2 > 0.
		\label{stepping-stone}
	\end{equation}
	Using (\ref{expan0}) we can write $f(x) = \sum_{j=0}^k \alpha_j x^j$,
	where the coefficients $\alpha_j$ are positive. Then
	\[ f''(x) + x^{-1} f'(x) = \alpha_1 x^{-1} + \sum_{j=2}^{k} j^2 \alpha_j\, x^{j-2}
	= \sum_{j=1}^{k} \big( j\sqrt{\alpha_j x^{j-2}}\big)^2,\]
	and $f(x) \geq \sum_{j=1}^k \big(\sqrt{\alpha_j\, x^j}\big)^2$.  
	Then the Cauchy--Schwarz inequality implies that
	\[ \big(f''(x) + x^{-1}\, f'(x)\big)\, f(x) \geq \left(\sum_{j=1}^k j\, \alpha_j x^{j-1}\right)^2 = \big( f'(x)\big)^2.
	\]
	Hence, to establish (\ref{stepping-stone}) it suffices to prove that
	\[ d(2k-d)\, \big( f''(x) + x^{-1}\, f'(x)\big)\, f(x) < f''(x)\, 
	  \big( 2(d-k)^2 + 2d(2k-d)\, f(x)\big),
	\]
	which we rearrange as
	\[ d(2k-d)\, x^{-1}\, f'(x)\, f(x) < f''(x)\, 
	\big( 2(d-k)^2 + d(2k-d)\, f(x)\big).
	\]
	But 
	\[ f''(x)\, \big( 2(d-k)^2 + d(2k-d)\, f(x)\big)\geq d(2k-d)\, f''(x)\, f(x)\]
	so the desired inequality holds if $x f''(x) > f'(x)$ for all $x\geq 1$.
	Since $f''$ is an increasing function on $x\geq 1$, by the Mean Value Theorem we
	have
	\[ f'(x) \leq f'(1) + (x-1)f''(x),\]
	so it suffices to check that $f'(1) \leq f''(1)$, again using the fact that
	$f''$ is monotonic increasing for $x\geq 1$.
	By (\ref{eq:f-derivs}),
	\[
	\frac{d(d-1)}{k^2}\big(f''(1) - f'(1)\big) = (k-1)^2 - (d-1)
	\geq \left(\frac{d+1}{2}\right)^2 - d + 1
	= \frac{(d-1)(d-5)}{4}.
	\]
	Hence the result holds.
\end{proof}

It remains to prove that (\ref{f-identity}) has no solution with $x\in (0,1)$.
This will follow from the next three lemmas.

\begin{lemma}
	\label{no-label}
	If $x\in (0,1)$, then 
	\[
	(1+\eta)f 
	<  \frac{d}{2(d-k)}\, f -  \frac{(2k-d)^2}{2d(d-k)}\, f^2.
	\]
\end{lemma}
\begin{proof} 
Since  $f < 1$ for $x\in (0,1)$, we get that
\begin{align*}
	 \frac{d}{2(d-k)} &- \frac{d}{d-k+ \sqrt{(d-k)^2 + d(2k-d)\, f}} \\
	&=  \frac{d^2(2k-d)\, f}{2(d-k)\big(d-k + \sqrt{(d-k)^2 + d(2k-d)\,f}\, \big)^2}\\
	&>   \frac{2k-d}{2(d-k)}\, f.
	\end{align*}
	Therefore
	\[
	\eta=
	\frac{2k-d}{d-k + \sqrt{(d-k)^2 + d(2k-d)\, f}}
	< \frac{2k-d}{2(d-k)} - \frac{(2k-d)^2}{2d(d-k)}\, f.
	\]
	The required bound follows after multiplying by $(2k-d)f/d$ and rearranging.
\end{proof}

Recall the function $g$ from (\ref{g-def}). 

\begin{lemma}\label{l:monotone}
	For all $t>0$, we have that 
	\[
	g(t g'' + g') > t (g')^2
	\]
	and hence the rational function $tg'/g$ is monotonically strictly increasing on $t>0$.
\end{lemma}
\begin{proof}
	Using the Cauchy-Schwarz inequality, we have
	\begin{align*}
		g(t g'' + g')  & > t \mu^{-2}
		\sum_{j=1}^{d-k} \binom{k}{j} \binom{d-k}{j} t^{j-1} \cdot
		\sum_{j =1}^{d-k} j^2  \binom{k}{j} \binom{d-k}{j} t^{j-1}
		\\
		&\geq t \mu^{-2} \left(  \sum_{j =1}^{d-k} j  \binom{k}{j} \binom{d-k}{j} t^{j-1}\right)^2 =  t (g')^2
	\end{align*}
	The first inequality is strict because we dropped the constant term in $g$.
\end{proof}

Now we are ready to prove the following bound.

\begin{lemma}\label{L3}
	Assume that $2d \geq 3k$ and \emph{(\ref{eq:P2})} holds.
	Then, for any $x\in(0,1)$,
	\[
  \frac{(1+x) f'}{k}> (1 + \eta) f.  
	\]
\end{lemma}
\begin{proof}
By Lemma~\ref{no-label},  it suffices to prove that for all $x\in (0,1)$,
\[ 
  \frac{(1+x) f'}{k} \geq \frac{d}{2(d-k)}\, f -  \frac{(2k-d)^2}{2d(d-k)}\, f^2.
  \]

Using (\ref{eq:ft-deriv}) we rewrite this inequality in terms of $g$: it
suffices to prove that for all $t > 1$, 
\[
\frac{1+ t^{-1}}{ k} \left(  - \frac{g'}{t^{k-2}} + \frac{k g}{t^{k-1}}\right) 
\geq   \frac{d}{2(d-k)}\, \frac{g}{t^{k}} - \frac{(2k-d)^2}{2d(d-k)}\, \frac{g^2}{t^{2k}}.
\]
This is equivalent to the inequality $A+ B \geq 0$, where $A$ is defined
in (\ref{def:A}) and
\begin{align*}
	B  &= B(t) =  \frac{(2k-d)^2\, g^2}{2d(d-k)\, t^k}.
\end{align*}
Next we state a sequence of claims which together will imply the result.

\bigskip

\noindent 
	\textbf{Claim 1.} $A'(t) \geq 0$ and  $A''(t)\geq 0$ for $t \geq 1$.

	\bigskip
	
	\noindent 
	\textbf{Claim 2.}  Either $B''(t)\geq 0$ for all $t\geq 1$, or there exists some $t_0>1$ such that 
	$B'(t) \geq 0$ for $t \geq t_0$ and  $B''(t)\geq 0$ for  $1 \leq t  \leq t_0$.
	
	\bigskip
	
	\noindent 
	\textbf{Claim 3.}  $A(1) + B(1) =0$ and $A'(1) + B'(1) \geq 0$.
	
	\bigskip
	
	First, we complete the proof based on these claims.  By a slight abuse of notation, set $t_0=\infty$ in the first case of Claim~2.
	Note that  
	\[ (A+B)' \geq A'(1) + B'(1) \geq 0\]
	for  $1 \leq t  \leq t_0$, since both $A$ and $B$ are convex on this interval.
	For $t\geq t_0$, we have $(A+B)' \geq 0$ since both $A$ and $B$ are increasing. 
	Therefore, $A+B$ is an increasing function on $[1,+\infty)$ and thus $A + B \geq A(1) + B(1)=0$, as required.   
Next, observe that Claim~1 follows from Lemma~\ref{l:Claim1}. Hence it remains to establish Claims~2 and~3.

	For Claim 2, direct calculations show that 
	\[ 
	\frac{t\, g'(t)}{g(t)} \rightarrow 
	\begin{cases} k(d-k)/d  &\text{ as $t\rightarrow 1$,}\\
		d-k & \text{ as $t\rightarrow \infty$},
	\end{cases}
	\]
	By our assumptions, we have $k(d-k)/d < k/2 \leq d-k$.
	Hence using Lemma \ref{l:monotone}, if $3k<2d$, then
	there is a unique point $t_0\in (1,\infty)$ such that
	\[ \frac{t_0\, g'(t_0)}{g(t_0)} = \frac{k}{2},
	\]
	while if $3k=2d$, then $t g'/g$ increases strictly with limit $k/2$ as 
		$t\to\infty$, and we set $t_0=\infty$.
	We calculate directly that
	\[ \left(\frac{g(t)^2}{t^{k}}\right)' = \frac{g}{t^{k+1}}\,
	\left( 2t\, g' - k\, g\right).
	\]
	This expression is strictly positive when $t>t_0$, by definition of $t_0$. This proves that $B'(t)\geq 0$ for $t \geq t_0$,  as $B$ is a positive multiple of $g^2/t^{k}$.
	Differentiating again gives
	\begin{align*}
		\left(\frac{g(t)^2}{t^{k}}\right)^{''} &= 
		\frac{2\, (g')^2 + 2g\, g''}{t^{k}} - \frac{4k g\, g'}{t^{k+1}}
		+ \frac{k(k+1)\, g^2}{t^{k+2}}\\
		&\geq \frac{4t^2\, (g')^2 - 2(2k+1)t\, g\, g' + k(k+1)\,g^2}{t^{k+2}}
		+ \frac{2g\, g' - 2t\, (g')^2 + 2t\, g\, g''}{t^{k+1}}\\
		&\geq \frac{1}{t^{k+2}}\, \left( k \, g - 2t\, g'\right)\left( (k+1) g- 2t \, g'\right).
	\end{align*}
	The final inequality follows since 
	the last term of the penultimate line
	(with denominator $t^{k+1}$) is nonnegative, by the first statement of
	Lemma \ref{l:monotone}.   Therefore if $1\leq t\leq t_0$, then,
	again by Lemma~\ref{l:monotone},
	\[ \frac{t\, g'(t)}{g(t)} \leq \frac{k}{2}\]
	and hence $B''(t) \geq 0$. 
	This completes the proof of Claim~2.
	
	\bigskip
	
	For Claim 3, recalling (\ref{eq:g-at-1}), we have 
	\[ 
	A(1) = -\frac{(2k-d)^2}{2d(d-1)} = - B(1).
	\]
	Using \eqref{e:c1}, we obtain  that 
	\[
	A'(1) = a_1 = \mu^{-1} k \binom{d-1}{k} c_1 \geq \frac{k(4k -d -2 - (2k-d)^2)}{2d^2}.
	\]
	Observe also  that 
	\[
	B'(1) = \frac{(2k-d)^2 }{2d(d-k)}\, \left( \frac{2  k (d-k)}{d} -k  \right)
	= -\frac{k (2k-d)^3}{2 d^2 (d-k)}.
	\]
	Therefore, by assumptions, we find that 
\[	 	A'(1) + B'(1) \geq \frac{k}{2d^2 (d-k)} \left( (d-k) (4k-d-2 - (2k-d)^2) - (2k-d)^3\right).\]
Now using (\ref{eq:P2}), the right hand side is nonnegative if 
$d-k > (2k-d)^3$, and this inequality holds when $3k\leq 2d$.  
This completes the proof of Claim~3 and the lemma.
\end{proof}

If $d\geq 5$ and $k=k_0(d)$, then $3k\leq 2d$ and (\ref{eq:P2}) holds.
Hence, by 
Lemma~\ref{L3} we conclude that 
(\ref{f-identity}) has no solution in $(0,1)$. 
Recalling Lemma~\ref{lem:boundary} and Lemma~\ref{L1-smalld}, this completes
the proof of Lemma~\ref{lem:global-max}
when $k=k_0(d)$ and $d\geq 5$. 

\subsection{Proof of Lemma \ref{lem:global-max}.  Part (II) for large $d$. }\label{s:large-d}

In this subsection, we  prove Lemma \ref{lem:global-max} for values of $k$ in the range
$k_0(d) < k\leq d/2 + \frac{\log d}{6}$ that are not covered in the previous subsection.
Thus, we can assume that 
$d \geq 404 =  \lceil e^6\rceil$.

\begin{lemma}\label{L2}
	Assume that $d\geq 404$ and  $d/2+1 < k \leq 2d/3$. Then
	\[
	\frac{(x+1)f'(x)}{k} < (1 + \eta(x) )f(x))
	\]
	 for all $x \in (1,2]$ such that $f(x) \leq k/(2k-d)^2$.
\end{lemma}
\begin{proof}
	Recall that $f(1)=1$ and hence $\eta(1) = (2k-d)/d$.
	Since both $f$ and $f'$ are strictly increasing functions of $x$, 
	applying the Mean Value Theorem to $\eta$ on the interval $[1,x]$ gives
	\begin{align*}
		\frac{\eta(1)- \eta(x)}{x-1}   
		&=  \frac{d(2k-d)^2\, f'(\xi)/\sqrt{(d-k)^2 + d(2k-d)\, f(\xi)}}
		  {\left(d-k+\sqrt{(d-k)^2 + d(2k-d)\, f(\xi)}\right)^2}\\
		&\leq \frac{(2k-d)^2\, f'(x)}{2kd} \leq
		\frac{f'(x)}{2d\, f(x)}
	\end{align*}
	for some $\xi\in (1,x)$. This uses the assumed upper bound on $f(x)$ for the final inequality.
	Hence to prove the lemma, it suffices to prove that
	\begin{equation}\label{eq_red}
		\frac{(x+1)f'}{k} < \frac{2k\, f}{d}  -  \frac{(x-1) f'}{2d}.
	\end{equation}
	Similarly to the proof of Lemma~\ref{L1},  using \eqref{expan1}, we expand as follows:
	\begin{equation}\label{expr}
		\begin{aligned}
 \frac{2k\, f}{d} &- \frac{(x+1)f'}{k}  -  \frac{ (x-1) f'}{2d} 
			\\ &= \mu^{-1}\sum_{j=0}^k 
			\binom{k}{j} \binom{d-j}{k-j} \,
		 \left( \frac{2k}{d} - \frac{j}{k}  - \frac{2(k-j)^2}{k(d-j)} - \frac{ j}{2d} \right) (x-1)^j.
		\end{aligned}
	\end{equation}
	Observe that, for all $0\leq j\leq k$,
	\begin{align}
		\frac{2k}{d}  - \frac{j}{k}  - \frac{2(k-j)^2}{k(d-j)} - \frac{j}{2d} 
		&= \frac{j}{d}\left(\frac{(2k-d-j)(d-k)}{k(d-j)} + \frac12 \right).
		\label{eq:rearrange}
	\end{align}

	Summing the terms of \eqref{expr} corresponding to $j\in \{0,k-1,k\}$ together and
	using (\ref{eq:rearrange}) gives
	\begin{align}
		\sum_{j\in \{1,k-1, k\}}
		&\binom{k}{j} \binom{d-j}{k-j} \, \frac{j}{d}\left(\frac{(2k-d-j)(d-k)}{k(d-j)} + \frac{1}{2}\right) \, (x-1)^j 
		\nonumber \\	 
		&> \frac{k}{2d}  \binom{d-1}{k-1}   (x-1)
		-   \frac{(k-1)(2d-3k)}{2d} \, (x-1)^{k-1} 
		-
		\frac{2d-3k}{2d}  (x-1)^k
		\nonumber \\
		&\geq \frac{k(x-1)}{2d} \left(
		\binom{d-1}{k-1}  -   (d-k+2)(2d-3k)   \right). \label{eq:sum-3}
	\end{align}
	Under our assumptions,  we have
	\[ \binom{d-1}{k-1} - (d-k+2)(2d-3k) > \frac{(d-2)}{2}\big( (d-1)- (d-k+2)\big)\geq 0,\]
	since $\binom{d-1}{k-1} > (d-1)(d-2)/2$ and $2d-3k\leq (d-2)/2$.
	Hence the sum over $j\in \{1,k-1,k\}$ in (\ref{eq:sum-3}) is strictly positive.
	
	Now,  for $2\leq j \leq k/2$, we will sum the terms corresponding to $(x-1)^j$ and
	$(x-1)^{k-j}$ in (\ref{expr}).  First observe that when $2\leq j\leq k/2$,
	\[ \frac{2k-d-j}{d-j} > -\frac{1}{2}\]
	and hence
	\begin{align*} \frac{(2k-d-j)(d-k)}{k(d-j)} + \frac{1}{2}
		&\geq -\frac{(d-k)}{2k} + \frac{1}{2} = \frac{2k-d}{2k} > 0.
	\end{align*}
	Next,
	\[         j \,\binom{k}{j}\, \binom{d-j}{k-j} \geq  (k-j) \, \binom{k}{k-j} \,
	\binom{d-k+j}{j},\] 
	so we can bound the term corresponding to $(x-1)^j$ from below as
	\begin{align*}
		\binom{k}{j} &\binom{d-j}{k-j}\frac{j}{d}\, (x-1)^j\, \left(\frac{(2k-d-j)(d-k)}{k(d-j)} + \frac{1}{2}\right) \\ 
		&\geq \frac{(k-j)}{d}\, \binom{k}{k-j}\, \binom{d-k+j}{j}\, (x-1)^{k-j}\, 
		\, \left(\frac{(2k-d-j)(d-k)}{k(d-j)} + \frac{1}{2}\right). 
	\end{align*}
	Finally, observe that 
	\begin{align*}
		\frac{(2k-d-j)(d-k)}{k(d-j)} &+ \frac{(2k-d-(k-j))(d-k)}{k\big(d-(k-j)\big)} + 1 \\
		&\geq 
		\frac{(2k-d)(d-k)}{kd} + \frac{(k-d)(d-k)}{k(d-k)} + 1 \\
		&= 1 - \frac{2(d-k)^2}{kd}\geq 0.
	\end{align*}
	Combining these facts shows that the sum of the $(x-1)^j$ and $(x-1)^{k-j}$ terms is
	nonnegative for $2\leq j\leq k/2$, completing the proof. 
\end{proof}

To complete this section we prove the following.

\begin{lemma}
	Suppose that $d\geq 404$ and $d/2 + 1 < k \leq d/2+\frac{\log d}{6}$.  
	If $x>1$, then
	\[ 
	\frac{(x+1)\, f'(x)}{k} < \big(1 + \eta(x)\big)\, f(x).
	\]
	\label{L1-bigd}
\end{lemma}

\begin{proof}
	By our assumptions, $k\leq 4d/7$ and hence 
	the assumptions of Lemma~\ref{L1} and Lemma~\ref{L2} hold and
	$\frac{5k-2d}{d-k}\leq 2$.  
	By Lemma~\ref{L1} and Lemma~\ref{L2}, if
	\begin{equation}
		\label{goal}
		f(x) \leq \frac{k}{(2k-d)^2} \quad \text{ for all $1 < x \leq \frac{5k-2d}{d-k}$},
	\end{equation}
	then the lemma holds. For the remainder of the proof, suppose that
	$1 < x \leq \frac{5k-2d}{d-k}$.
	It follows from \eqref{eq_red} that  
	\[
	\frac{(x+1) f'}{k} \leq \frac{2k\, f}{d},
	\]
	and hence
	$
	(\log  f)' < \frac{2k^2}{d\, (x+1)}. 
	$
	Integration gives 
	\begin{align*}
		f(x) \leq  \left(1+ \frac{x-1}{2}\right)^{2k^2/d} &\leq 
		\exp\left(\frac{k^2(x-1)}{d}\right)\\
		&\leq 
		 \exp\left(\frac{3k^2(2k-d)}{d(d-k)}\right).
	\end{align*}
	Hence
	\[
	\log(f(x)) \leq \frac{(d + \nfrac{1}{3}\log d)^2}{2d(d-\nfrac{1}{3}\log d)}\, \log d\leq \dfrac{3}{5}\log d,
	\]
	while
	\[
	\frac{k}{(2k-d)^2} \geq 
	\frac{9d}{2(\log d)^2} = \frac{9 d^{3/5}}{50}\, \left(\frac{d^{1/5}}{\log(d^{1/5})}\right)^2 \geq \frac{9e^2}{50}\, d^{3/5}\geq d^{3/5}.
	\]
	This shows that (\ref{goal}) holds, completing the proof.
\end{proof}

As observed earlier, if $d\geq 404$ and $d/2 < k < d/2 + \frac{\log d}{6}$,
then $3k\leq 2d$.  Furthermore, (\ref{eq:P2}) holds as 
$\frac{\log d}{6} < \frac{1 + \sqrt{d-1}}{2}$ for $d\geq 404$.
Hence the proof of Lemma~\ref{lem:global-max} is completed in this
case 
by combining Lemma~\ref{lem:boundary}, Lemma~\ref{L3} and Lemma~\ref{L1-bigd}, for $k>k_0(d)$,
and using the results of Section~\ref{ss:proof-B-k0} when $k=k_0(d)$.

\subsection*{Acknowledgements}

We are very grateful to the referees for their insightful comments which have
improved this paper.
In particular, we would like to thank one referee for pointing out that using
$\alpha^{\text{1-RSB}}(d)$ an upper bound on $\alpha^*(d)$ could lead to
improvements in our results.
We also thank the mathematical research institute MATRIX in Creswick, Australia where part of this research was performed.

Michelle Delcourt's research was supported by NSERC under Discovery Grant No. 2019-04269. The research of Catherine Greenhill and Mikhail Isaev was supported by the Australian
Research Council Discovery Project DP250101611.
Bernard Lidick{\' y}'s research was supported in part by NSF grant DMS-2152490 and the Scott Hanna Professorship.  Luke Postle's research was partially supported by NSERC under Discovery Grant No. 2019-04304.

\end{document}